\documentclass{article}

\usepackage{tikz}

\usepackage{graphicx}
\usepackage{babel}
\usepackage{mathrsfs}
\usepackage{amsmath,amsfonts,amssymb,amsthm,cases}
\usepackage{ucs}
\usepackage[utf8x]{inputenc}
\usepackage[T1]{fontenc}
\usepackage{comment}
\usepackage{enumerate}
\usepackage{textcomp,url}
\usepackage{eucal}
\usepackage[noBBpl]{mathpazo}

\makeatletter

\renewcommand{\p@enumii}{}
\def\@enum@{\list{\csname label\@enumctr\endcsname}          {\usecounter{\@enumctr}\def\makelabel##1{
			\normalfont\ignorespaces\emph{{##1}~}}
		\setlength{\labelsep}{3pt}
		\setlength{\parsep}{0pt}
		\setlength{\itemsep}{0pt}
		\setlength{\leftmargin}{0pt}
		\setlength{\labelwidth}{0pt}
		\setlength{\listparindent}{\parindent}
		\setlength{\itemsep}{0pt}
		\setlength{\itemindent}{0pt}
		\topsep=3pt plus 1pt minus 1 pt}}
\makeatother
\renewcommand{\epsilon}{\ensuremath{\varepsilon}}
\renewcommand{\phi}{\ensuremath{\varphi}}

\renewcommand{\to}{\ensuremath{\longrightarrow}}

\allowdisplaybreaks

\makeatletter
\def\@map#1#2[#3]{\mbox{$#1 \colon\thinspace #2 \to #3$}}
\def\map#1#2{\@ifnextchar [{\@map{#1}{#2}}{\@map{#1}{#2}[#2]}}
\makeatother

\newtheoremstyle{theoremm}{}{}{\itshape}{}{\scshape}{.}{ }{}
\theoremstyle{theoremm}
\newtheorem{theorem}{Theorem}
\newtheorem{lemma}{Lemma}
\newtheorem{proposition}{Proposition}
\newtheorem{corollary}{Corollary}
\newtheorem*{definition}{Definition}
\newtheoremstyle{reptheorem}{}{}{\itshape}{}{\scshape}{}{ }{\thmname{#1}\mathrm{#3}}
\theoremstyle{reptheorem}

\newtheoremstyle{remark}{}{}{}{}{\scshape}{.}{ }{}
\theoremstyle{remark}
\newtheorem{remark}{Remark}

\newtheorem{example}{Example}
\newtheoremstyle{comment}{}{}{\bfseries}{}{\bfseries}{:}{ }{}
\theoremstyle{comment}

\newcommand{\rethA}[1]{Theorem~A}

\begin{document}

\title{Skew-symmetric augmented matrices and a characterization of virtual doodles}
	\author{OSCAR~OCAMPO~\\
		Universidade Federal da Bahia,\\
		Departamento de Matem\'atica - IME,\\
		Av. Milton Santos~S/N~CEP:~40170-110 - Salvador - BA - Brazil.\\
		e-mail:~\texttt{oscaro@ufba.br}\vspace*{4mm}\\ JOS\'E~GREGORIO~RODR\'IGUEZ-NIETO\\Departamento de Matem\'aticas - Facultad de Ciencias,\\Universidad Nacional de Colombia,\\Carrera 65 N. 59A-110, Medell\'in-Colombia.\\e-mail:~\texttt{jgrodrig@unal.edu.co} \vspace*{4mm}\\ OLGA~PATRICIA~SALAZAR-D\'IAZ\\Departamento de Matem\'aticas - Facultad de Ciencias,\\Universidad Nacional de Colombia,\\Carrera 65 N. 59A-110, Medell\'in-Colombia.\\e-mail:~\texttt{opsalazard@unal.edu.co} 
        }
  
	\maketitle
\begin{abstract}
   In this paper, we present a brief overview of the concept of doodles from the perspective of J.~S.~Carter’s work on classifying immersed curves 
   and the work of J. Carter, S. Kamada, and M.~Saito on stable equivalence of knots on surfaces and virtual knot cobordisms.
   We use the homology intersection number and the work of G. Cairns and D.~Elton on the Gauss word problem  
   to introduce the concept of skew-symmetric augmented matrices for determining whether a virtual doodle is non-classical. We also provide a  characterization of the virtualization of classical doodles.
\end{abstract}

\section{Introduction}

Doodles were first introduced by Fenn and Taylor \cite{FT} as embeddings of a collection of circles in $S^{2}$  with the condition that there are no triple or higher intersection points. These embeddings are considered up to an equivalence relation that allows the introduction or removal of monogons and bigons. Later, Khovanov \cite{K} expanded this concept by allowing self intersections of curves on surfaces. In addition, Khovanov introduced an algebraic structure he referred to as twin groups, which we reinterpret here as planar braids. Doodles can also be understood as planar projections of knots and links in the plane, where the third Reidemeister move is disallowed, and the first and second Reidemeister moves are flattened. This creates an analogous relationship to the well-known correspondence: "(virtual) knots and links" relate to "(virtual) braid groups" \cite{Ka}. The new correspondence in this context is: "(virtual) doodles" correspond to "(virtual) planar braid groups" \cite{G0, NS}.

Due to the similarity between doodles on surfaces and the concept of normal curves, we present a topological approximation of doodle diagrams on oriented and compact surfaces, similar to those provided by Carter in \cite{Ca} and by Carter, Kamada, and Saito in \cite{CaKaSa}. This approach allows us to apply the work of Cairns and Elton, as found in \cite{CaEl} and \cite{CaEl2}, on the Gauss word problem to define the skew-symmetric augmented matrix associated to a doodle diagram, which is a doodle invariant. The main property of this invariant is that it is null for the set of doodle diagrams on the sphere. Therefore, skew-symmetric augmented matrices can determine whether a doodle is not equivalent to a doodle diagram on the sphere $S^{2}$. These latter doodles are commonly referred to as \textit{classical doodles}. Additionally, we introduce the concept of virtualization of doodles on the sphere and study the behavior of skew-symmetric augmented matrices with respect to this operation. Thus, we provide sufficient conditions to determine, in many cases, when the virtualization of a crossing point in a classical doodle is non-classical.

This paper is organized as follows: In Section~\ref{sec:doodles}, we provide a brief overview of the concept of doodles in terms of stable R-equivalence, which is a slight modification of the stable equivalence introduced in \cite{Ca}, \cite{CaKaSa}, \cite{CaEl}, \cite{CaEl2}, and \cite{Ro2}. We prove that this definition is equivalent to the one given in \cite{BaFeKaKa}. Moreover, we define minimal doodles and compute the minimum genus of a doodle diagram, such genus is given in terms of the number of components of any regular neighborhood of the doodle diagram. We also use the homology intersection number of a collection of curves, defined in \cite{CaEl}, on the ambient surface of a doodle diagram to characterize certain types of doodle diagrams. In Section~\ref{sec:skew}, we use intersection homology to introduce and study the definition of skew-symmetric augmented matrices, providing some important properties. We prove that Kishino's doodle is non-classical. In Section~\ref{sec:virtual}, we introduce the concept of virtualization of classical doodles and examine the behavior of augmented matrices with respect to this operation. In particular, we present a brief characterization of the virtualization of classical doodles.      

\subsection*{Acknowledgments}

The first named author would like to thank Eliane Santos, all HCA staff, Bruno Noronha, Luciano Macedo, Marcio Isabela, Andreia de Oliveira Rocha, Andreia Gracielle Santana, Ednice de Souza Santos, and Vinicius Aiala for their valuable help since July  2024. O.O. was partially supported by National Council for Scientific and Technological Development - CNPq through a \textit{Bolsa de Produtividade}, project number 305422/2022~-~7.

\section{Doodles}\label{sec:doodles}

In this paper the word \textit{surface} means an oriented and compact surface. An \textit{oriented doodle diagram} of $n$-components or, for simplicity, \textit{doodle diagram}, is a tuple $(\Sigma,D)$, where $\Sigma$ is a surface and $D\subset \Sigma$ is the image of an orientation-preserving continuous map $\gamma\colon  \coprod_{i=1}^{n}S_{i}^{1}\rightarrow \Sigma $ from $n$ disjoint circles $S^{1}$ to $\Sigma$, such that the number of self intersections, called \textit{crossing points of $D$}, is finite, no more than three points have the same image and the intersections are transverse. The surface $\Sigma$ is called the \textit{ambient surface} of $D$. In the case that $\Sigma=S^{2}$ is the $2$-sphere, we say that $(S^{2},D)$ is a \textit{classical doodle diagram}. The restriction of $\gamma$ to each circle $S_{i}^{1}$ is called the $i$th-component of $D$.
\begin{definition}\label{geotopic}
We say that two doodle diagrams $(\Sigma_{1},D_{1})$ and $(\Sigma_{2},D_{2})$ are \textit{geotopic}, denoted by $(\Sigma_{1},D_{1}) \stackrel{e}{\sim}  (\Sigma_{2},D_{2})$ if there exist a surface $\Sigma_{3}$ and orientation-preserving embeddings $f_{i}\colon  \Sigma_{i}\rightarrow \Sigma_{3}$, $i=1,2$, such that $f_1(D_1)=f_2(D_2)$. 
\end{definition}
The geotopy relation is reflexive and symmetric, but not transitive. For example, if $\Sigma_{1}=S^{1}\times [0,1]$ (a cylinder), $\Sigma_{2}=\{(x,y,z)\in \mathbf{R}^{3}\mid x^{2}+y^{2}+z^{2}+1 \}$ (a sphere) and  $\Sigma_3=S^{1}\times S^{1}$ the torus surface, then 
$$(\Sigma_2,D_{2}:=\{(x,y,z)\in \Sigma_2\mid z=0\})\stackrel{e}{\sim} (\Sigma_1,D_{1}:=S^{1}\times \{1/2\})\stackrel{e}{\sim} (\Sigma_3,D_{3}:=S^{1}\times \{(1,0)\}),$$
but $(\Sigma_2,D_{2})$ is not geotopic to $(\Sigma_3,D_{3})$.
\begin{definition}\label{geotopic2}
  Two doodle diagrams $(\Sigma,D)$ and $(\Sigma',D')$ are set to be \textit{stably equivalent}, denoted $(\Sigma,D)\sim(\Sigma',D')$ if there exists a finite collection of doodle diagrams $\{(\Sigma_i,D_i)\}_{i=1}^{n}$, such that 
\[
(\Sigma,D)\stackrel{e}{\sim} (\Sigma_1,D_1)\stackrel{e}{\sim} \cdots \stackrel{e}{\sim}(\Sigma_n,D_n)\stackrel{e}{\sim} (\Sigma',D').
\]
\end{definition}
The stable equivalency relation is, in fact, an equivalence relation on the set of doodles diagrams.   

\begin{definition}[Surgery on ambient surfaces]\label{surgery}
Let $(\Sigma, D)$ be a doodle diagram, and let ${\gamma_{j}}{j=1}^{k}$ be a disjoint collection of $k$ simple closed curves in $\Sigma$ that neither intersect the boundary of $\Sigma$ nor the doodle diagram $D$. Consider regular neighborhoods $V_{D}^{\Sigma}$ and $\{V_{\gamma_i}^{\Sigma}\}{i=1}^{k}$ of $D$ and the curves $\gamma_1, \dots, \gamma_{k}$, respectively, with disjoint closures. Now, remove the neighborhoods $V_{\gamma_1}^{\Sigma}, \dots, V_{\gamma_k}^{\Sigma}$ from $\Sigma$ and glue $2k$ disks via their boundaries to the new boundaries of the surface. The resultant surface, denoted by $h_{\gamma_1, \dots, \gamma_{k}}^{-}(\Sigma)$, is called the \textit{surgery} of $\Sigma$ along the curves $\{\gamma_{j}\}_{j=1}^{k}$, and the corresponding doodle diagram $(h_{\gamma_1, \dots, \gamma_{k}}^{-}(\Sigma), D)$ is called a surgery of $(\Sigma, D)$ along the curves $\{\gamma_{j}\}_{j=1}^{k}$.
\end{definition}
Observe that, when we remove the regular neighborhood $\{V_{\gamma_i}^{\Sigma}\}_{i=1}^{k}$ we obtain a closed and orientable new surface, $\Sigma_{\gamma_{1},\cdots,\gamma_k}$, which has $2k$ new components of boundary different from the boundary of the original surface $\Sigma$. If we denote these components by $\{\gamma_{i}^{+}, \gamma_{i}^{-}\}_{i=1}^{k}$, then, the surface $h_ {\gamma_1,\cdots, \gamma_{k}}^{-}(\Sigma)$ is obtained by gluing closed disks $\mathbf{d}_{i}^{+}, \mathbf{d}_{i}^{-}$ to each component $\gamma_{i}^{+}, \gamma_{i}^{-}$, $i=1,\cdots, k$ of $\Sigma_{\gamma_{1},\cdots,\gamma_k}$. Thereby, the curves $\gamma_{i}^{+}, \gamma_{i}^{-}$ become, respectively, the border of simply connected closed regions $\Gamma_{i}^{+}, \Gamma_{i}^{-}$ in  $h_ {\gamma_1,\cdots, \gamma_{k}}^{-}(\Sigma)$, $i=1,\cdots,k$. Since, $\gamma_{1},\cdots, \gamma_{k}$ are simply closed curves in $\Sigma$, the closure of the regular neighborhoods $V_{\gamma_1}^{\Sigma},\cdots, V_{\gamma_k}^{\Sigma}$ can be seen as thin tubes of the form $\gamma_{i}\times [0,1]$, where $\gamma_{i}^{+}:=\gamma_{i}\times \{0\}$ and $\gamma_{i}^{-}=\gamma_{i}\times \{1\}$, or, as it is well known in the literature, $1$-handles. Thus, the surgery of $\Sigma$ along the curves $\{\gamma_{j}\}_{j=1}^{k}$ is a sequence of the elimination of the $1$-handles $\gamma_{i}\times [0,1]$, $i=1,\cdots,k$. So, surgery is, in fact, a sequence of \textit{$1$-handle elimination}. Hence, we have the following recursive form of a surgery  
\[
h_ {\gamma_1,\cdots, \gamma_{k}}^{-}(\Sigma)=h_{\gamma_{1}}^{-}(h_{\gamma_{2}}^{-}( \cdots (h_{\gamma_{k-1}}^{-}( h_{\gamma_{k}}^{-})(\Sigma))\cdots)).
\]
The reciprocal is also true. For example, suppose that we want to eliminate an $1$-handle in $\Sigma$. Since, any $1$-handle in $\Sigma$ is a section of the form $\gamma\times [0,1]$, where $\gamma$ is a simple closed curve in $\Sigma$; then, the elimination of the $1$-handle can be seen as the surgery along the curve $\gamma \times \{\frac{1}{2}\}$.

\begin{definition}[$1$-handle addition]\label{1hadleaddition}
Let $(\Sigma,D)$ be a surface doodle diagram and let $\gamma^{+}$ and $\gamma^{-}$ be two disjoint  null-homologue curves in $\Sigma$ bordering two compact and disjoint simply connected regions $\Gamma^{+}$ and $\Gamma^{-}$, respectively, such that $D\cap (\Gamma^{+}\cup \Gamma^{-})=\emptyset$ and $\partial(\Sigma)\cap (\Gamma^{+}\cup \Gamma^{-})=\emptyset$. Now, remove the interior of $\Gamma^{+}$ and $\Gamma^{-}$ and attach an $1$-handle, $S^{1}\times [0,1]$,  by its boundary to the new boundary components $\gamma^{+}$  and $\gamma^{-}$ of $\Sigma$. The resultant surface, denoted by $h^{+}_{\gamma^{+},\gamma^{-}}(\Sigma)$, is called the $1$-handle addition to $\Sigma$ along the curves $\gamma^{+}$ and $\gamma^{-}$.
\end{definition}

In general, if we have a collection of $2k$ null-homologue curves $\gamma_1^{+},  \gamma_{1}^{-}, \cdots, \gamma_k^{+},  \gamma_{k}^{-}$, bordering $2k$ disjoint compact and simply connected regions $\Gamma_{1}^{+}, \Gamma_{1}^{-}, \cdots,\Gamma_{k}^{+}, \Gamma_{k}^{-} $, respectively, in $\Sigma$, such that these regions do not intersect neither  the doodle diagram $D$ nor the boundary of $\Sigma$. Then $h^{+}_{\gamma_1^{+},\gamma_1^{-},\cdots,\gamma_k^{+},\gamma_k^{-}}(\Sigma)$ is the surface obtained by gluing $k$ $1$-handles $\{(S^{1}\times [0,1])_{i}\}_{i=1}^{k}$ to $\Sigma$ recursively as follows: 
\[
h^{+}_{\gamma_1^{+},\gamma_1^{-},\cdots,\gamma_k^{+},\gamma_k^{-}}(\Sigma)=h^{+}_{\gamma_1^{+},\gamma_1^{-}}(h^{+}_{\gamma_2^{+},\gamma_2^{-}}( \cdots h^{+}_{\gamma_{k-1}^{+},\gamma_{k-1}^{-}}( h^{+}_{\gamma_k^{+},\gamma_k^{-}}(\Sigma))\cdots ))
\]
 and $(h^{+}_{\gamma_1^{+},\gamma_1^{-},\cdots,\gamma_k^{+},\gamma_k^{-}}(\Sigma),D)$ is called a \textit{$1$-handles addition} to $(\Sigma,D)$ along the paired curves $\{\gamma_i^{+},  \gamma_{i}^{-}\}_{i=1}^{k}$. 

 \begin{remark}
 With the notation of Definition \ref{1hadleaddition}. If $\gamma:=S^{1}\times \{\frac{1}{2}\}$ is a specific meridian of the attached $1$-handle $S^{1}\times [0,1]$ to $\Sigma$ along the curves $\gamma^{+},\gamma^{-}$, then $h_{\gamma}^{-}( h^{+}_{\gamma^{+},\gamma^{-}}(\Sigma))$ is homeomorphic to $\Sigma$. Reciprocally, if $\gamma$ is a simple and closed curve in $\Sigma$ and $\gamma\times [0,1]$ is the removed $1$-handle used to construct $h^{-}_{\gamma}(\Sigma)$, then $h^{+}_{\gamma^{+},\gamma^{-}}( h^{-}_{\gamma}(\Sigma))$ is homeomorphic to $\Sigma$, where $\gamma^{+}=\gamma\times \{0\}$ and $\gamma^{-}=\gamma\times \{1\}$. In order to simplify the notation, we use 
 \begin{center}
 $h^{-}(\Sigma)=h_ {\gamma_1,\cdots, \gamma_{k}}^{-}(\Sigma)$ and  $h^{+}(\Sigma)=h^{+}_{\gamma_1^{+},\gamma_1^{-},\cdots,\gamma_k^{+},\gamma_k^{-}}(\Sigma)$. 
 \end{center}
 Thereby, $h^{+}(\Sigma)$ and $h^{-}(\Sigma)$ can be seem as inverse surgeries one of the other.
 \end{remark}
Next theorem relates the concepts of valid surgery and stable R-equivalence.

\begin{theorem}\label{handletheorem}
Let $(\Sigma,D)$ be a surface doodle digram. Then, for any surgeries $h^{-}$ and $h^{+}$ of $(\Sigma,D)$, 
\[(h^{-}(\Sigma),D) \sim (\Sigma,D)\sim (h^{+}(\Sigma),D).\]
\end{theorem}
\begin{proof}
We first prove the case in which $h^{-}=h_{\gamma}^{-}$ the surgery is along the curve $\gamma$, the rest of the proof comes from the fact that $h_ {\gamma_1,\cdots, \gamma_{k}}^{-}(\Sigma)=(h_{\gamma_{1}}^{-}\circ \cdots \circ h_{\gamma_{k}}^{-})(\Sigma)$. Since the inclusion function $i\colon  \Sigma \setminus Int(V_{\gamma}^{\Sigma}) \rightarrow \Sigma$ and the identity function $id\colon  \Sigma \rightarrow \Sigma$ are orientation-preserving embeddings, and $i(D)=Id(D)$, then $(\Sigma\setminus Int(V_{\gamma}^{\Sigma}),D) \stackrel{R}{\sim} (\Sigma,D)$. In a similar way, we also have that the inclusion function $i\colon  \Sigma \setminus Int(V_{\gamma}^{\Sigma}) \rightarrow h^{-}_{\gamma}(\Sigma)$ and the identity function $id\colon  h^{-}_{\gamma}(\Sigma) \rightarrow h^{-}_{\gamma}(\Sigma)$ are orientation-preserving  embeddings, and that $i(D)=Id(D)$, then $(\Sigma\setminus V_{\gamma}^{\Sigma},D) \stackrel{R}{\sim} (h^{-}_{\gamma}(\Sigma),D)$. Thus,
\[
(\Sigma,D)\stackrel{R}{\sim} (\Sigma\setminus V_{\gamma}^{\Sigma} ,D)\stackrel{R}{\sim} (h^{-}_{\gamma}(\Sigma),D).
\]
Therefore, $(h^{-}(\Sigma),D) \sim (\Sigma,D)$. In a similar way, we get the proof of $(h^{+}(\Sigma),D) \sim (\Sigma,D)$.
\end{proof}

The proof of the following lemma comes in the same fashion that the one of the previous theorem.
\begin{lemma}
   Let $(\Sigma,D)$ be a doodle diagram. Then  $(\Sigma,D)$ is stable equivalent to $(S_{g}(\Sigma),D)$, where $S_g(\Sigma)$ is the closed surface of genus $g$, obtained by gluing disks  to each component of the boundary $\partial(\Sigma)$ of $\Sigma$. 
\end{lemma}\label{glue}
Let us suppose that $\gamma_1, \dots, \gamma_k$ are the components of the boundary of $\Sigma$, and let $\mathbf{d}_1, \dots, \mathbf{d}_k$ be the respective disks that we will glue to each component of $\partial(\Sigma)$ to obtain $S{g}(\Sigma)$. Then,
\begin{equation}\label{addhandle}
S_{g}(\Sigma) = h^{+}_{\gamma_1, \partial(\mathbf{d}_1)} \left( h^{+}_{\gamma_2, \partial(\mathbf{d}_2)} \left( \cdots \left( h^{+}_{\gamma_{k-1}, \partial(\mathbf{d}_{k-1})} \left( h^{+}_{\gamma_k, \partial(\mathbf{d}_k)} (\Sigma) \right) \right) \cdots \right) \right).
\end{equation}

\begin{theorem} \label{defequiv}
Let $(\Sigma,D)$ and $(\Sigma',D')$ be two doodles diagram. Then, $(\Sigma,D)$ and $(\Sigma',D')$ are stably equivalent if and only if one of $(S_{g}(\Sigma),D)$ and $(S_{d}(\Sigma'),D')$ can be changed  into the other by a finite sequence of the following operations
\begin{enumerate}
\item surgeries of type $h^{+}$ and $h^{-}$, 
\item elimination of connected components of the ambient surface disjoint of the doodle diagram and 
\item orientation-preserving homeomorphism of the ambient surface.
\end{enumerate}
\end{theorem}
\begin{proof} Let us denote by $\approx$ the equivalence relation generated by the transformations described in  literals $(a)$-$(b)$, and  let $(\Sigma,D)$ be a doodle diagram and a regular neighborhood $O_{D}^{\Sigma}$ of $D$ in the interior of $\Sigma$. Let $\gamma_{1},\cdots, \gamma_{k}$ be the component of $\partial(O_{D}^{\Sigma})$. Then, these curves are simply-closed curves in $\Sigma$ disjoint from the doodle diagram $D$. From the fact that the regular neighborhoods $V_{\gamma_1}^{\Sigma},\cdots V_{\gamma_{k}}^{\Sigma}$ of $\gamma_{1},\cdots, \gamma_{k}$, given in Definition \ref{surgery}, do not meet the doodle diagram $D$ and overlap  $O_{D}^{\Sigma}$; then, after removing these regular neighborhoods, the resultant surface $\Sigma_{\gamma_1,\cdots,\gamma_k}$ has a component $V_{D}^{\Sigma}$ that is a regular neighborhood of $D$ contained in $O_{D}^{\Sigma}$. Therefore, when we glue the respective discs along the new boundary components $\gamma_{1}^{+},\gamma_1^{-},\cdots, \gamma_{k}^{+},\gamma_{k}^{-}$ of $\Sigma_{\gamma_1,\cdots,\gamma_k}$, we have that $h^{-}_{\gamma_1,\cdots,\gamma_k}(\Sigma)=S_{g'}(V_{D}^{\Sigma}) \coprod S_{g_1}(V_{\gamma_1}^{\Sigma}) \coprod \cdots \coprod S_{g_k}(V_{\gamma_k}^{\Sigma})$. Thus,  

\begin{equation}\label{elihand}
(S_{g'}(\Sigma),D)\approx  (S_{g}(O_{D}^{\Sigma}),D),
\end{equation} under transformations of type $(b)$.

Now, we suppose that we have two doodle diagrams $(\Sigma_{1},D_{1})$ and $(\Sigma_{2},D_{2})$, with  $(\Sigma_{1},D_{1}) \stackrel{e}{\sim}  (\Sigma_{2},D_{2})$. Then, there exist a compact and orientable surface $\Sigma_{3}$ and orientation-preserving embedding $f_{i}\colon  \Sigma_{i}\rightarrow \Sigma_{3}$, $i=1,2$, such that $f_1(D_1)$ and $f_2(D_2)$ are ambient isotopic. Then, 
\begin{enumerate}
\item[(1)] There exist an orientation-preserving homeomorphism $\varphi\colon  \Sigma_3 \rightarrow \Sigma_3$, such that $\varphi(f_1(D_1))=f_2(D_2)$. Therefore,  $(\Sigma_3,f_1(D_1))\approx (\Sigma_3,f_2(D_2))$. So, from equation (\ref{addhandle}), 
$$(S_h(\Sigma_3),f_1(D_1))\approx (\Sigma_3,f_1(D_1))\approx (\Sigma_3,f_2(D_2)) \approx (S_h(\Sigma_3),f_2(D_2)).$$
\end{enumerate}
Let $V_{D_i}^{\Sigma_i}\subset \Sigma_i$ be a regular neighborhood of $D_i$, $i=1,2$. Since, $f_i(V_{D_i}^{\Sigma_i})=V_{f(D_i)}^{\Sigma_3}$ is a regular neighborhood of $f_i(D_i)$ in $\Sigma_3$, $i=1,2$, from Equation (\ref{elihand}),

\begin{enumerate}
\item[(2)]  $(S_{g_i'}(\Sigma_i),D_i)\approx (S_{g_i}(V_{D_i}^{\Sigma_i}),D_i)$, $i=1,2$, and 
\item[(3)] $(S_{h}(\Sigma_3),f_i(D_i))\approx \left(S_{h_i}\left(V_{f_i(D_i)}^{\Sigma_3}\right),f_i(D_i)\right)$, $i=1,2$. Note that $h_i=g_i$, $i=1,2.$
\end{enumerate}
On the other hand, from the fact that $f_1$ and $f_2$ are orientation-preserving  embedding, then 
\begin{enumerate}
\item[(4)] $(V_{D_i}^{\Sigma_i},D_i)\approx (V_{f_i(D_i)}^{\Sigma_3},f_i(D_i))$, hence $(S_{g_i}(V_{D_i}^{\Sigma_i}),D_i)\approx (S_{h_i}(V_{f_i(D_i)}^{\Sigma_3}),f_i(D_i))$, $i=1,2$.
\end{enumerate}
Thus, we have the following sequence,
\[
\begin{array}{ccl}
(S_{g_1'}(\Sigma_1),D_1)&\approx& (S_{g_1}(V_{D_1}^{\Sigma_1}),D_1)\approx (S_{h_1}(V_{f_1(D_1)}^{\Sigma_3}),f_1(D_1))\approx (S_{h}(\Sigma_3),f_1(D_1))\\
&&\\
&\approx&(S_{h}(\Sigma_3),f_2(D_2))\approx \left(S_{h_2}\left(V_{f_2(D_2)}^{\Sigma_3}\right),f_2(D_2)\right)\approx (S_{g_2}(V_{D_2}^{\Sigma_2}),D_2)\\
&&\\
&\approx&(S_{g_2'}(\Sigma_2),D_2) .
\end{array}
\]
In this way, if $(\Sigma,D)\sim (\Sigma',D')$, then there exist a sequence of surface doodles $\{(\Sigma_i,D_i)\}_{i=1}^{k}$, such that 
\[
(\Sigma,D)\stackrel{R}{\sim} (\Sigma_1,D_1)\stackrel{R}{\sim} \cdots \stackrel{R}{\sim}(\Sigma_k,D_k)\stackrel{R}{\sim} (\Sigma',D'),
\]
then
\[
(S_g(\Sigma),D)\approx (S_{g_1}(\Sigma_1),D_1)\approx \cdots \approx (S_{g_k}(\Sigma_k),D_k)\approx (S_{g'}(\Sigma'),D').
\]

Reciprocally, let $(\Sigma,D)$ and $(\Sigma',D)$ be two surface doodle diagrams such that $(S_{g}(\Sigma),D) \approx (S_{g'}(\Sigma'),D)$. Then there exist a sequence of doodle diagrams $\{(\Sigma_i,D_i)\}_{i=1}^{k}$, such that 

\[
(S_g(\Sigma),D)\approx (\Sigma_1,D_1)\approx \cdots \approx (\Sigma_k,D_k)\approx (S_{g'}(\Sigma'),D'),
\]  
and each equivalence in the sequence is one of the $(a)$-$(d)$. Thereby, if $\approx$ corresponds to the type $(a)$, then, from Theorem \ref{handletheorem}, this implies $\sim$. It is clear that any homeomorphism is an embedding function, so $(c)$ and $(d)$ implies $\sim$. Suppose that $\approx$ is given by elimination of superfluous components. Then, we have that $i\colon  \Sigma\rightarrow \Sigma \coprod \Omega$, the inclusion function and the identity $Id\colon  \Sigma \coprod \Omega \rightarrow \Sigma \coprod \Omega$, satisfy the condition of the definition of geotopy relation. Thus, $(\Sigma \coprod \Omega, D )\sim (\Sigma,D)$. 
\end{proof}

We say that a doodle diagram $(S,\gamma)$ is \textit{minimal} if the ambient surface $S$ is a connected, closed, orientable and compact surface with the minimum genus among all the doodle diagrams $(\Sigma,D)$ stably equivalent to $(S,\gamma)$.
    \begin{proposition}
        Let $(\Sigma,D)$ be a doodle diagram. Then there exists a minimal doodle diagram $(S,\gamma)$ stably equivalent to $(\Sigma,D)$. 
    \end{proposition}
\begin{proof}
    Let $(\Sigma, D)$ be a doodle diagram, and let $\mathscr{M}$ be the collection of all connected, closed, orientable, and compact surfaces $S$ for which there exists a doodle diagram $(S, \sigma)$ stably equivalent to $(\Sigma, D)$. The set $\mathscr{M}$ is not empty because $S_{g}(\Sigma)$ belongs to $\mathscr{M}$, see Lemma \ref{glue}. We now consider the set $M = \{ \text{genus}(S) \mid S \in \mathscr{M} \}$. Then, by the well-ordering principle, $M$ has a minimal element, say $N_{0}$. Therefore, there exists a doodle diagram $(S, \gamma)$ with $\text{genus}(S) = N_0$, stably equivalent to $(\Sigma, D)$. Thereby, $(S, \gamma)$ is the desired minimal doodle diagram.  
\end{proof}    
 A formula to compute the number $N_0$ is given in \cite[Corollary~3.5]{Ro2}, here we give a short review of such construction. Let $(\Sigma,D)$ be a doodle diagram and choose any  regular neighborhood $V_{D}^{\Sigma}\subset \Sigma$  of $D$ disjoint from the boundary of $\Sigma$. Then, $(S_{g'}(V_{D}^{\Sigma}),D)$ is a doodle diagram stably equivalent to $(\Sigma,D)$ and $genus(S_{g'}(V_{D}^{\Sigma}))\leq genus(\Sigma)$.    
If $S_{g'}(V_{D}^{\Sigma})$ is a non-connected surface, it is because $D=D_1 \coprod D_2 \coprod \cdots \coprod D_k$ has more than one connected component. Therefore, $V_{D}^{\Sigma}$ can be decomposed as $V_{D}^{\Sigma}=V_{D_1}^{\Sigma}\coprod V_{D_2}^{\Sigma}\coprod \cdots \coprod V_{D_k}^{\Sigma}$, where  $V_{D_1}^{\Sigma},\cdots,V_{D_k}^{\Sigma}$ are disjoint regular connected neighborhoods of the components $D_1,\cdots,D_k$ of $D$, respectively.  Then,  

\[S_{g'}(V_{D}^{\Sigma})=S_{g_1}(V_{D_1}^{\Sigma}) \coprod S_{g_2}(V_{D_2}^{\Sigma}) \coprod \cdots \coprod S_{g_k}(V_{D_k}^{\Sigma}).\] 
Moreover, the genus $g_i$ of $S_{g_i}(V_{D_i}^{\Sigma})$ satisfies the equality
\begin{equation*}
g_i=genus(S_{g_{i}}(V_{D_i}^{\Sigma}))=\frac{m_i-2-\mid\partial(V_{D_{i}}^{\Sigma}) \mid}{2}, i=1,2,\cdots, k, 
\end{equation*}
where $m_i$ is the number of crossing points of $D_i$ and $\mid\partial(V_{D_i}^{\Sigma}) \mid$ is the number of components of the boundary of $V_{D_i}^{\Sigma}$. Hence fore, the genus $g'$ of $S_{g'}(V_{D}^{\Sigma})$ also satisfies
\begin{equation}\label{planeequ}
g'=genus(S_{g'}(V_{D}^{\Sigma}))=\frac{m-2-\mid\partial(V_{D}^{\Sigma}) \mid}{2},
\end{equation}
where $m$ is the number of crossing points of $D$. Because, $(S_{g_i}(V_{D_i}^{\Sigma}),D_i)$ is a doodle diagram, $i=1,2,\dots,k$, we write 
\[
(S_{g'}(V_{D}^{\Sigma}),D)=(S_{g_1}(V_{D_1}^{\Sigma}),D_1)\coprod (S_{g_2}(V_{D_2}^{\Sigma}),D_2)\coprod \cdots \coprod (S_{g_k}(V_{D_k}^{\Sigma}),D_k),
\]
and we say that $(S_{g_1}(V_{D_1}^{\Sigma}),D_1),\cdots, (S_{g_k}(V_{D_k}^{\Sigma}),D_k) $ are the connected components of $(S_{g'}(V_{D}^{\Sigma}),D)$. In the case that $k=1$, we say that $(S_g(V_{D}^{\Sigma}),D)$ is a connected doodle diagram.

\begin{remark}
If we denote the surface $S_{g'}(V_{D}^{\Sigma})$ by $M_D$, then, up to homeomorphism, $M_{D}$ does not depend on either the ambient surface $\Sigma$ or the chosen regular neighborhood $V_{D}^{\Sigma}$. The surface $M_D$ is related in \cite{Ro2} with the Carter surface of $D$, see \cite{Ca} for more details.
\end{remark}

\begin{example}
    A construction of a regular neighborhood is illustrated in Figure \ref{regular}. 
      \begin{figure}[ht]
	\begin{center}
		\includegraphics[scale=0.25]
		{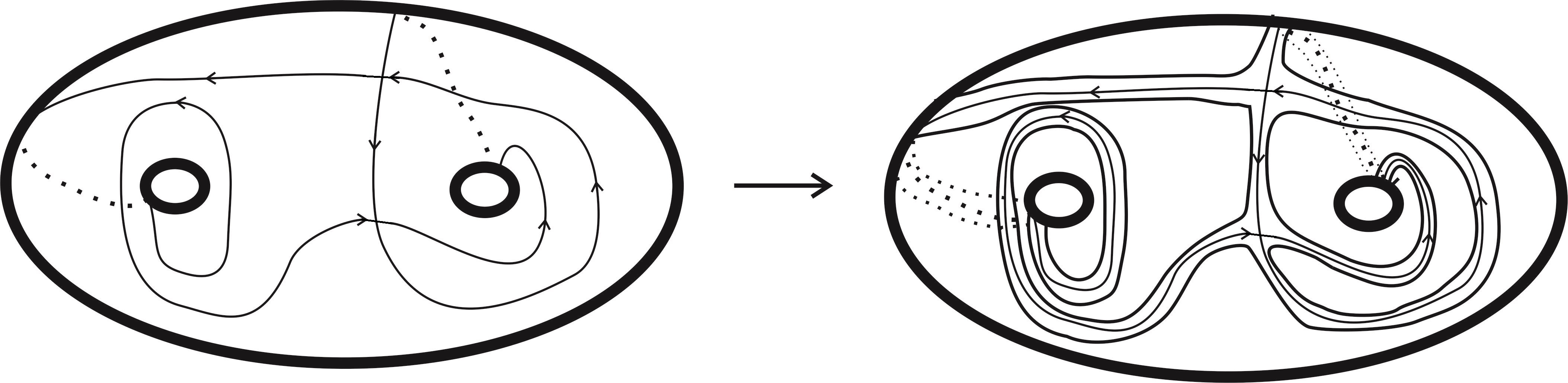}
  \caption{An example of a construction of a regular neighborhood of a doodle diagram}
  \label{regular}
	\end{center}
\end{figure}  
In this case, the number of components of the regular neighborhood of the doodle diagram is $2$. So, the minimal genus of the normal curve is $genus=\frac{2+2-2}{2}=1$. Thereby, the Carter surface of $K$ is the torus surface.
\end{example}
A doodle diagram $(\Sigma,D)$ is called \textit{almost classical} if it is geotopy equivalent to a classical doodle diagram.
    \begin{proposition}
        A doodle diagram $(\Sigma,D)$ is almost classical, if and only if  $\mid\partial(V_{D}^{\Sigma}) \mid=m-2$, where $m$ is the number of crossing points of $D$. 
    \end{proposition}
\begin{proof}
    A direct consequence of the equation (\ref{planeequ}).
\end{proof}

Let $(\Sigma, D)$ be a one-component doodle diagram, and label the crossing points of $D$ with the letters $a_1, \ldots, a_n$. For each $a_i$, we consider the two closed curves, $D_{a_i}$ and $\widetilde{D}_{a_i}$, where $D_{a_i} \subset D$ (or $\widetilde{D}_{a_i}\subset D$) is constructed by following $D$ starting from a non-crossing point.  If when $D$ travels the crossing point $a_i$, it goes from right to left (resp. left to right), then  instead of continuing along $D$, we turn to the left (resp. right) and follow $D$ until return to the starting point. The constructed curve is denoted by $D_{a_i}$ (resp. $\widetilde{D}_{a_i}$). The curves $D_{a_i}$ and $\widetilde{D}_{a_i}$ are well defined and do not depend on the starting point. Moreover, $D=D_{a_i}+D_{a_i}$ in $H_{1}(\Sigma,\mathbf{Z})$, see Figure \ref{primicurves}.

\begin{figure}[ht]
    \centering
    \tikzset{every picture/.style={line width=0.75pt}} 

\begin{tikzpicture}[x=0.75pt,y=0.75pt,yscale=-1,xscale=1]

\draw [line width=1.5]    (284.67,213.3) -- (284.67,178.3) ;
\draw [shift={(284.67,175.3)}, rotate = 90] [color={rgb, 255:red, 0; green, 0; blue, 0 }  ][line width=1.5]    (14.21,-4.28) .. controls (9.04,-1.82) and (4.3,-0.39) .. (0,0) .. controls (4.3,0.39) and (9.04,1.82) .. (14.21,4.28)   ;
\draw [line width=1.5]    (284.67,175.3) -- (327.67,175.3) ;
\draw [shift={(330.67,175.3)}, rotate = 180] [color={rgb, 255:red, 0; green, 0; blue, 0 }  ][line width=1.5]    (14.21,-4.28) .. controls (9.04,-1.82) and (4.3,-0.39) .. (0,0) .. controls (4.3,0.39) and (9.04,1.82) .. (14.21,4.28)   ;
\draw  [line width=3] [line join = round][line cap = round] (284.33,175.3) .. controls (284.33,175.19) and (284.33,175.08) .. (284.33,174.97) ;
\draw  [line width=3] [line join = round][line cap = round] (285,174.97) .. controls (284.84,174.97) and (284.52,175.25) .. (284.67,175.3) ;
\draw [line width=1.5]    (117,205.5) -- (117,124.5) ;
\draw [shift={(117,121.5)}, rotate = 90] [color={rgb, 255:red, 0; green, 0; blue, 0 }  ][line width=1.5]    (14.21,-4.28) .. controls (9.04,-1.82) and (4.3,-0.39) .. (0,0) .. controls (4.3,0.39) and (9.04,1.82) .. (14.21,4.28)   ;
\draw [line width=1.5]    (71,167.5) -- (160,167.5) ;
\draw [shift={(163,167.5)}, rotate = 180] [color={rgb, 255:red, 0; green, 0; blue, 0 }  ][line width=1.5]    (14.21,-4.28) .. controls (9.04,-1.82) and (4.3,-0.39) .. (0,0) .. controls (4.3,0.39) and (9.04,1.82) .. (14.21,4.28)   ;
\draw  [line width=3] [line join = round][line cap = round] (116.67,167.5) .. controls (116.67,167.39) and (116.67,167.28) .. (116.67,167.17) ;
\draw  [line width=3] [line join = round][line cap = round] (117.33,167.17) .. controls (117.18,167.17) and (116.85,167.45) .. (117,167.5) ;
\draw [line width=1.5]    (278.45,167.65) -- (278.65,124.5) ;
\draw [shift={(278.67,121.5)}, rotate = 90.27] [color={rgb, 255:red, 0; green, 0; blue, 0 }  ][line width=1.5]    (14.21,-4.28) .. controls (9.04,-1.82) and (4.3,-0.39) .. (0,0) .. controls (4.3,0.39) and (9.04,1.82) .. (14.21,4.28)   ;
\draw [line width=1.5]    (232.67,167.5) -- (275.45,167.64) ;
\draw [shift={(278.45,167.65)}, rotate = 180.19] [color={rgb, 255:red, 0; green, 0; blue, 0 }  ][line width=1.5]    (14.21,-4.28) .. controls (9.04,-1.82) and (4.3,-0.39) .. (0,0) .. controls (4.3,0.39) and (9.04,1.82) .. (14.21,4.28)   ;
\draw  [line width=3] [line join = round][line cap = round] (278.33,167.5) .. controls (278.33,167.39) and (278.33,167.28) .. (278.33,167.17) ;
\draw  [line width=3] [line join = round][line cap = round] (279,167.17) .. controls (278.84,167.17) and (278.52,167.45) .. (278.67,167.5) ;

\draw (119.4,170.1) node [anchor=north west][inner sep=0.75pt]   [align=left] {$\displaystyle a_{i}$};
\draw (243.7,106) node [anchor=north west][inner sep=0.75pt]   [align=left] {$\displaystyle D_{a_{i}}$};
\draw (101.67,217.17) node [anchor=north west][inner sep=0.75pt]   [align=left] {$\displaystyle ( a)$};
\draw (330.6,154.5) node [anchor=north west][inner sep=0.75pt]  [font=\scriptsize] [align=left] {$\displaystyle a_{i}$};
\draw (46.5,64.5) node [anchor=north west][inner sep=0.75pt]   [align=left] {$ $};
\draw (269.33,220.57) node [anchor=north west][inner sep=0.75pt]   [align=left] {$\displaystyle (b)$};
\draw (319.7,146.4) node [anchor=north west][inner sep=0.75pt]   [align=left] {$\displaystyle \widetilde{D}$};

\end{tikzpicture}
    \caption{Construction of the primitive curve $D_{a_i}$}
    \label{primicurves}
\end{figure}
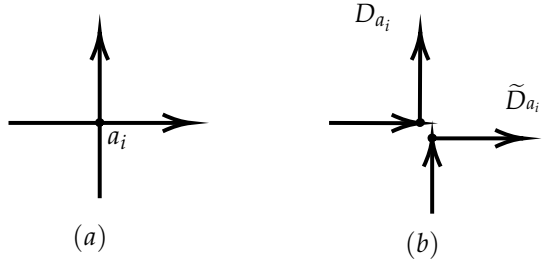

 \begin{theorem}\cite{CaEl}\label{CaElt}
     Let $(\Sigma,D)$ be a one component doodle diagram. Then, the homology group $H_1(\Sigma,\mathbf{Z})$ is generated by the homology classes $\varphi_{D},\varphi_{a_1},\ldots,\varphi_{a_n}$ represented by the curves $D,D_{a_1},\dots,D_{a_n}$, respectively, if and only if $(\Sigma,D)$ is a minimal doodle diagram.
 \end{theorem}
We now review the definition of the intersection number; for more details, see \cite{FoMa}, \cite{Fu} and \cite[Section~2.4]{Ro2}. Let $\gamma_1$ and $\gamma_2$ be two transverse, oriented, and generically immersed curves on a surface $\Sigma$ with $\gamma_1 \cap \gamma_2=\{c_1,\ldots,c_k\}$. We choose a non crossing point  $x$ on $\gamma_1$. We follow the curve $\gamma_{1}$ writing down an ordered set $\gamma_1 \nearrow \gamma_2$ with the crossing labels $c_j$'s we meet on the way, with the convention that we add $c_{i}^{+1}$ (or $c^{-1}_{i}$) if when $\gamma_1$ travels the crossing point $c_i$, $\gamma_2$ goes from left to right (right to left). The \textit{intersection pairing number}, denoted by $\left\langle\gamma_1,\gamma_2 \right\rangle$, is defined as the sum of the superscript of the elements of $\gamma_1 \nearrow \gamma_2$. It is well known that all homology classes $\varphi_1$ and $\varphi_2$ in $H_{1}(\Sigma,\mathbf{Z})$ are represented by transverse and oriented generically immersed curves $\gamma_1$ and $\gamma_2$, respectively, and that the number $\left\langle \varphi_1,\varphi_2  \right\rangle=\left\langle\gamma_1,\gamma_2 \right\rangle$ is well defined, and called the \textit{homology intersection number} between $\varphi_1$ and $\varphi_2$. Homology intersection defines an skew-symmetric bi-linear map  $\left\langle \ \, \ \ \right\rangle \colon H_{1}(\Sigma,\mathbf{Z})\times H_{1}(\Sigma,\mathbf{Z}) \rightarrow \mathbf{Z}$, in other words, for every $\varphi_i\in H_{1}(\Sigma,\mathbf{Z})$, $i=1,2,3$ and $t\in \mathbf{Z}$, we have
\begin{itemize}
    \item $\left\langle \varphi_1,\varphi_2 \right\rangle=-\left\langle \varphi_2,\varphi_1 \right\rangle$,
    \item $\left\langle \varphi_1,\varphi_2 + t \varphi_3\right\rangle=\left\langle \varphi_1,\varphi_2\right\rangle+t\left\langle \varphi_1, \varphi_3\right\rangle$.
\end{itemize}
Thus, we have that $\left\langle \varphi,\varphi \right\rangle=0$ for every $\varphi \in  H_{1}(\Sigma,\mathbf{Z})$.
\begin{theorem}
     A one-component doodle diagram $(\Sigma,D)$ is almost classical if and only if the  homology intersection number between the elements of the subset $\{ \varphi_{D},\varphi_{a_1},\ldots,\varphi_{a_n} \}$ of the first homology group $H_{1}(\Sigma,\mathbf{Z})$ is zero.
\end{theorem}
\begin{proof}
    Let $V_{D}^{\Sigma}\subset \Sigma$ be a regular neighborhood of $D$ in $\Sigma$, then $\{ \varphi_{D},\varphi_{a_1},\ldots,\varphi_{a_n} \} \subset H_{1}(S_{g}(V_{D}^{\Sigma})$. Since $(S_{g}(V_{D}^{\Sigma}),D)$ is a minimal doodle diagram,  from Theorem \ref{CaElt}, $\{ \varphi_{D},\varphi_{a_1},\ldots,\varphi_{a_n} \}$ is a generating set for the first homology group  $H_{1}(S_{g}(V_{D}^{\Sigma}),\mathbf{Z})$ of $\Sigma$. Thus, $genus(S_{g}(V_{D}^{\Sigma}))=0$ if and only if 
    \begin{equation}\label{albe}
    \left\langle \varphi_{D},\varphi_{a_i} \right\rangle=0 \text{ and } \left\langle \varphi_{a_i},\varphi_{a_j} \right\rangle=0 \text{ for all } i,j.
    \end{equation}
    Therefore, $(\Sigma,D)$ is almost classical if and only if it satisies the equations (\ref{albe}). 
\end{proof}

\begin{definition}
  Let $(\Sigma,D)$ and $(\Sigma',D')$ be two doodles diagram. Then, $(\Sigma,D)$ and $(\Sigma',D')$ are said to be  stably Reidemeister equivalent or, for simplicity, stably R-equivalent if and only if one of them can be changed into the other one by a finite sequence of the following operations
\begin{enumerate}
\item stably equivalence and 
\item monogon and bigon type moves, see figure \ref{f4}.
\begin{figure}[ht]
	\begin{center}
		\includegraphics[scale=0.6]
		{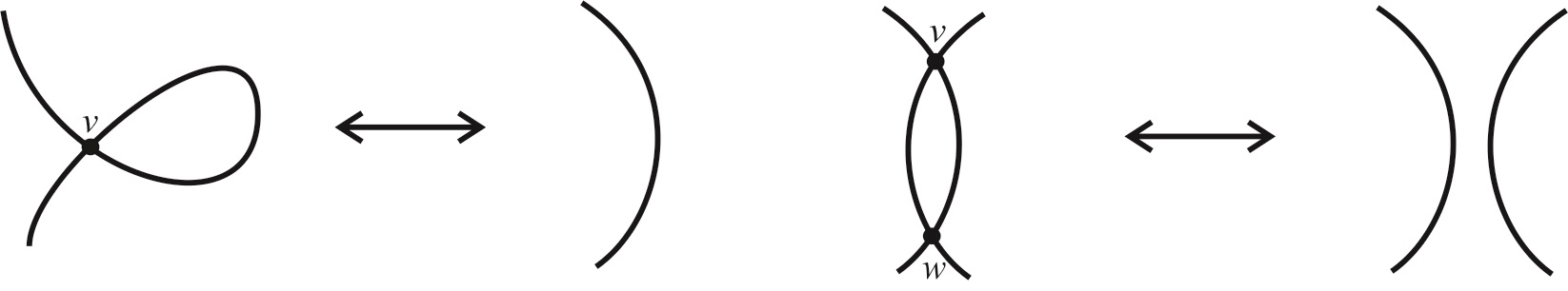}
		\caption{monogon and bigon type moves}
		\label{f4}
	\end{center}
\end{figure}  
\end{enumerate}
\end{definition}

It is not hard to prove that the stably R-equivalent is in fact an equivalence relation. An \textit{oriented doodle} or \textit{doodle} is defined as an equivalence class of a doodle diagram under the stably R-equivalence. A doodle is called \textit{classical} if its equivalence class has a classical doodle diagram. It is clear that if a doodle is almost classical then it is classical, the reciprocal is not always true. A doodle is \textit{trivial} if its equivalence class has the doodle diagram without crossing points. The \textit{recognition doodle problem} consists in determining when a doodle is classical or not. 

\section{Skew-symmetric augmented matrices}\label{sec:skew}

For any natural number $n>0$,  a \textit{skew-symmetric augmented matrix} of dimension $n$ is a matrix of the form $(B\mid A)$, where $B$ is a skew-symmetric matrix of order $n\times n$ and $A$ is a vector of order $n\times 1$. We define the skew-symmetric augmented matrix of order zero as the empty matrix $(\ \ \mid \ \ )$. The set of skew-symmetric augmented matrices is denoted by $\mathbf{Skew}$.

\begin{definition}
    Let $(B\mid A)$ and $(B'\mid A')$ be two skew-symmetric augmented matrices. We say that $(B'\mid A')$ is a permutation of $(B\mid A)$ if there exists a permutation matrix $P$ such that $(B'\mid A')=(PBP^{T}\mid PA)$, where $P^{T}$ denotes the transpose of the matrix $P$.      
\end{definition}
It is  worth to point out here that if $(B\mid A)$ is a skew-symmetric augmented matrix, then for every permutation matrix $P$, $(PBP^{T}\mid PA)$ is also a skew-symmetric augmented matrix. Thereby, permutation is a well-defined relation on the set $\mathbf{Skew}$. 
\begin{proposition}\label{perp}
   Let $(\Sigma,D)$ be a one-component doodle diagram with $n$ crossing points labeled with the letters $a_1,\ldots,a_n$, and let $\varphi_{D},\varphi_{a_1},\ldots,\varphi_{a_n}$  the homology classes in the first homology group $H_{1}(\Sigma,\mathbf{Z})$ of $\Sigma$  represented by the curves $D,D_{a_1},\dots,D_{a_n}$, respectively. Then, the augmented matrix $(\beta(D)\mid \alpha(D))$ defined as

\begin{equation}
    (\beta(D)\mid \alpha(D))=\left( 
	\begin{array}{ccc|c}
	\left\langle \varphi_{a_1},\varphi_{a_1} \right\rangle & \cdots & \left\langle \varphi_{a_1},\varphi_{a_n} \right\rangle & \left\langle \varphi_{a_1},\varphi_D \right\rangle\\ 
	\vdots & \ddots & \vdots & \vdots\\ 
	\left\langle \varphi_{a_n},\varphi_{a_1} \right\rangle & \cdots & \left\langle \varphi_{a_n},\varphi_{a_n} \right\rangle & \left\langle \varphi_{a_n},\varphi_D \right\rangle 
	\end{array}%
 \right),
\end{equation}
is a skew-symmetric augmented matrix. Moreover, if $(\beta(D)\mid \alpha(D))$ and $(\beta'(D)\mid \alpha'(D))$ are two  skew-symmetric augmented matrices obtained by two labelings of the crossing points of $D$, then one of these matrices is a permutation of the other one.
\end{proposition}
\begin{proof}
    Let us denote the set of crossing points of $D$ by $\rtimes(D)$. Two labelings of $\rtimes(D)$ can then be defined using bijective maps $f\colon  \rtimes(D)\rightarrow \{a_1,\ldots,a_n\}$  and $g\colon  \rtimes(D)\rightarrow \{b_1,\ldots,b_n\}$. Suppose these labelings define skew-symmetric augmented matrices  $(\beta(D)\mid \alpha(D))$ and $(\beta'(D)\mid \alpha'(D))$, respectively. Thus, if $\sigma\in Sym(n)$ is the permutation given by $\sigma(i)=j$ if and only if $(g\circ f^{-1})(a_i)=b_j$, then, the permutation matrix $P$ associated with $\sigma$ satisfies $(P\beta(D)P^{T}\mid P\alpha(D))=(\beta'(D)\mid \alpha'(D))$. 
\end{proof}
It is worth noting that surgeries of types $h^{+}$ and $h^{-1}$, as well as the elimination of redundant connected components of the ambient surface, do not affect the doodle diagram. Besides, the homology intersection number remains unchanged under orientation-preserving homeomorphisms. Thus, we have the proof of the following lemma. 
\begin{lemma}\label{stable}
   Let $(\Sigma,D)$ be a one-component doodle diagram. Up to permutations, the skew-symmetric augmented matrix $(\beta(D)\mid \alpha(D))$ is invariant under stable equivalency relation.  Therefore, the doodle diagram  $(\Sigma,D)$ is almost classical if and only if $(\beta(D)\mid \alpha(D) )$ is the null matrix.
\end{lemma}
Let us introduce the following elementary  transformations on the set of skew-symmetric augmented matrices. It is important to note that these transformations encodes the monogon and bigon movements in terms of augmented matrices.
\begin{definition}[Elementary transformations]
	Let $(B \mid A)$ be a skew-symmetric augmented matrix. Let $B^{(j)}$ denote the $j^{th}$ column of the matrix $B$ and $A_r$ the $r^{th}$ component of the vector $A$.
	
	\begin{enumerate}
	    \item If there exists $j\in\{1,\ldots,n\}$ such that $a_{j}=0$ and $(B^{(j)}
	=(0)_{n\times1}$ or $B^{(j)}=A)$, then we reduce the augmented matrix $(B \mid A)$ by simultaneously eliminating its $j$th column and the $j$th row. This type of reduction is called \textbf{reduction of type 1}.
 \item If there are $r,t\in\{1,\ldots,n\}$ such that 
	$A_{r}=-A_{t}$ and $B^{(r)}+B^{(t)}=A$, then we reduce the augmented matrix $(B \mid A)$ by simultaneously eliminating its  columns $r$ and $s$ and its rows $r$ and $s$. This type of reduction is called \textbf{reduction of type 2}.
	\end{enumerate}   
 The inverse of a reduction of type 1 is called \textbf{extension of type 1} and the inverse of a reduction of type 2 is called \textbf{extension of type 2}. 
\end{definition}
Since the elementary transformations are well defined on the set $\mathbf{Skew}$, we are able to introduce the following definition. 

\begin{definition}
    Let $(B\mid A)$ and $(B'\mid A')$ be two skew-symmetric augmented matrices. We say that they are $S$-equivalent, denoted by $\asymp$ if there exists a finite collection $\{(B_{i}\mid A_{i})\}_{i=0}^{n+1}$ of skew-symmetric augmented matrices such that $(B_{0}\mid A_{0})=(B\mid A)$, $(B_{n+1}\mid A_{n+1})=(B'\mid A')$ and $(B_{i+1}\mid A_{i+1})$ is obtained from $(B_i\mid A_i)$ by using only a permutation or an elementary transformation.   
\end{definition}
It is not hard to verify that the relation previously defined is an equivalence relation. The $S$-equivalence class of the skew-symmetric augmented matrix $(B\mid A)$ is denoted by $[B\mid A]$. 

\begin{theorem}
    Let $(\Sigma,D)$ be an one-component doodle diagram. Then the map $\lambda\colon  \mathbf{\mathcal{DC}}\rightarrow \mathbf{Skew}$, where $\mathbf{\mathcal{DC}}$ denotes the set of one-component doodle diagrams and  $\lambda((\Sigma,D))=(\beta(D)\mid \alpha(D))$, extends to a function $\Lambda\colon  \mathbf{\mathcal{DC}} /_{\stackrel{e}{\sim} } \rightarrow \mathbf{Skew}/_{\asymp}$. We will use the notation $\Lambda([(\Sigma,D)])=[\beta(D)\mid \alpha(D)]$.
\end{theorem}
\begin{proof}
From Proposition \ref{perp} and Lemma \ref{stable} we only have to prove the theorem for monogon and bigon type moves.  In fact, Let $(\Sigma,D)$ be an one-component doodle diagram with crossing points labeled with the letters $a_1,\ldots,a_n$ and let $\varphi_{D},\varphi_{a_1},\ldots,\varphi_{a_n}$ be the homology classes in $H_1(\Sigma,\mathbb{Z})$ represented by $D,D_{a_1},\ldots,D_{a_n}$.

\begin{enumerate} 

\item From Proposition \ref{perp}, we do not loose generality if we assume that $(\Sigma, D)$ has a monogon at the crossing point $a_n$. Let $(\Sigma,K)$ be the doodle diagram obtained by eliminating  such monogon of $(\Sigma,D)$ and let us  denote by $\psi_K$, $\psi_{a_1}$,\ldots,$\psi_{a_{n-1}}$ the homology classes in $H_1(\Sigma,\mathbf{Z})$ represented by the closed curves $K$, $K_{a_1}$,\ldots,$K_{a_{n-1}}$, respectively. Then, $D$ is homologous to $K$. Moreover, $\psi_{a_i}$ is homologous to $\varphi_{a_i}$, for all $i=1,\ldots,n-1$. Therefore,
\[
\left\langle \varphi_{a_i},\varphi_D \right\rangle=\left\langle \psi_{a_i},\psi_K \right\rangle \text{ and } \left\langle \varphi_{a_k},\varphi_{a_t} \right\rangle=\left\langle \psi_{a_k},\psi_{a_t} \right\rangle, \text{ for all } i,k,t=1,\ldots,n-1.
\]
On the other hand, we have two cases, $D_{a_n}$ is null-homologous, see Figure \ref{primimon}-$(a)$ or $D_{a_n}$  is homologous to $D$, see Figure \ref{primimon}-$(b)$. In any way, $\left\langle \varphi_{a_n},\varphi_D \right\rangle=0$. 
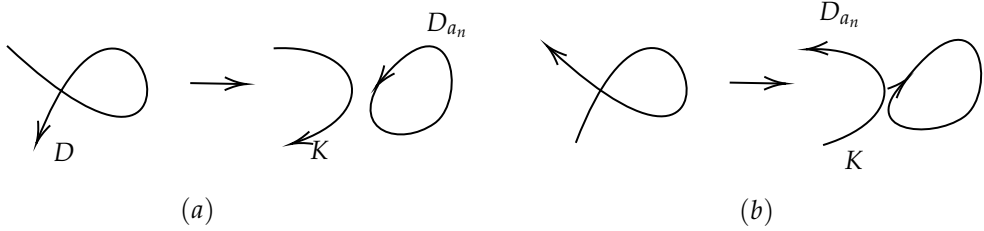
\begin{figure}[ht]
    \centering
  \tikzset{every picture/.style={line width=0.75pt}} 

\begin{tikzpicture}[x=0.75pt,y=0.75pt,yscale=-1,xscale=1]

\draw    (48.93,42.13) .. controls (165.73,152.53) and (113.73,-39.87) .. (63.73,91.33) ;
\draw [shift={(63.73,91.33)}, rotate = 290.86] [color={rgb, 255:red, 0; green, 0; blue, 0 }  ][line width=0.75]    (10.93,-3.29) .. controls (6.95,-1.4) and (3.31,-0.3) .. (0,0) .. controls (3.31,0.3) and (6.95,1.4) .. (10.93,3.29)   ;
\draw    (141.33,60.93) -- (167.67,61.41) ;
\draw [shift={(169.67,61.44)}, rotate = 181.03] [color={rgb, 255:red, 0; green, 0; blue, 0 }  ][line width=0.75]    (10.93,-3.29) .. controls (6.95,-1.4) and (3.31,-0.3) .. (0,0) .. controls (3.31,0.3) and (6.95,1.4) .. (10.93,3.29)   ;
\draw    (183.4,43.4) .. controls (227.88,40.96) and (239.7,76.22) .. (193.36,91.28) ;
\draw [shift={(191.93,91.73)}, rotate = 342.86] [color={rgb, 255:red, 0; green, 0; blue, 0 }  ][line width=0.75]    (10.93,-3.29) .. controls (6.95,-1.4) and (3.31,-0.3) .. (0,0) .. controls (3.31,0.3) and (6.95,1.4) .. (10.93,3.29)   ;
\draw    (236.17,61.83) .. controls (222.67,85.08) and (247.06,93.46) .. (264.49,81.27) .. controls (281.74,69.21) and (274,13.83) .. (236.24,61.2) ;
\draw [shift={(235.08,62.67)}, rotate = 307.8] [color={rgb, 255:red, 0; green, 0; blue, 0 }  ][line width=0.75]    (10.93,-3.29) .. controls (6.95,-1.4) and (3.31,-0.3) .. (0,0) .. controls (3.31,0.3) and (6.95,1.4) .. (10.93,3.29)   ;
\draw    (322.73,43.52) .. controls (437.27,150.37) and (385.54,-39.46) .. (335.79,91.08) ;
\draw [shift={(320.99,41.88)}, rotate = 43.39] [color={rgb, 255:red, 0; green, 0; blue, 0 }  ][line width=0.75]    (10.93,-3.29) .. controls (6.95,-1.4) and (3.31,-0.3) .. (0,0) .. controls (3.31,0.3) and (6.95,1.4) .. (10.93,3.29)   ;
\draw    (413.39,60.68) -- (439.72,61.16) ;
\draw [shift={(441.72,61.19)}, rotate = 181.03] [color={rgb, 255:red, 0; green, 0; blue, 0 }  ][line width=0.75]    (10.93,-3.29) .. controls (6.95,-1.4) and (3.31,-0.3) .. (0,0) .. controls (3.31,0.3) and (6.95,1.4) .. (10.93,3.29)   ;
\draw    (453.95,43.81) .. controls (497.22,42.51) and (507.77,77.66) .. (460.49,92.23) ;
\draw [shift={(451.96,43.9)}, rotate = 356.86] [color={rgb, 255:red, 0; green, 0; blue, 0 }  ][line width=0.75]    (10.93,-3.29) .. controls (6.95,-1.4) and (3.31,-0.3) .. (0,0) .. controls (3.31,0.3) and (6.95,1.4) .. (10.93,3.29)   ;
\draw    (501.59,60.15) .. controls (477.11,85.28) and (514.41,89.97) .. (531.49,78.02) .. controls (548.92,65.83) and (540.84,9.45) .. (502.08,59.42) ;
\draw [shift={(503.17,58.58)}, rotate = 136.15] [color={rgb, 255:red, 0; green, 0; blue, 0 }  ][line width=0.75]    (10.93,-3.29) .. controls (6.95,-1.4) and (3.31,-0.3) .. (0,0) .. controls (3.31,0.3) and (6.95,1.4) .. (10.93,3.29)   ;

\draw (416.5,117.5) node [anchor=north west][inner sep=0.75pt]   [align=left] {$\displaystyle ( b)$};
\draw (134.5,117) node [anchor=north west][inner sep=0.75pt]   [align=left] {$\displaystyle ( a)$};
\draw (469.83,93.42) node [anchor=north west][inner sep=0.75pt]   [align=left] {$\displaystyle K$};
\draw (456.72,18.42) node [anchor=north west][inner sep=0.75pt]   [align=left] {$\displaystyle D_{a_{n}}$};
\draw (258.67,23.67) node [anchor=north west][inner sep=0.75pt]   [align=left] {$\displaystyle D_{a_{n}}$};
\draw (200.33,88.33) node [anchor=north west][inner sep=0.75pt]   [align=left] {$\displaystyle K$};
\draw (70.67,89.67) node [anchor=north west][inner sep=0.75pt]   [align=left] {$\displaystyle D$};
\end{tikzpicture}
    \caption{Generating primitive curves in a monogon}
    \label{primimon}
\end{figure}
Moreover, 
\begin{enumerate}
 \item If $D_{a_n}$ is null-homologous, then $\left\langle \varphi_{a_i},\varphi_{a_n} \right\rangle=0$, for all $i$. Hence,
  \begin{equation*}
    (\beta(D)\mid \alpha(D))=\left( 
	\begin{array}[c]{cc|c}
	\beta(K) &0 & \alpha(K)\\
 0&0&0
	\end{array}%
 \right).
\end{equation*}
    \item If $D_{a_n}$ is homologous to $D$, then $ \left\langle \varphi_{a_i},\varphi_{a_n} \right\rangle=\left\langle \phi_{a_i},\phi_D \right\rangle = \left\langle \psi_{a_i},\psi_{K} \right\rangle$, for all  $i=1,\ldots,n-1.$ Thus,
\begin{equation*}
    (\beta(D)\mid \alpha(D))=\left( 
	\begin{array}[c]{cc|c}
	\beta(K) &\alpha(K) & \alpha(K)\\
 -\alpha(K)^{T}&0&0
	\end{array}%
 \right).
\end{equation*}
\end{enumerate}
From $(i)$ and $(ii)$, we have that  $(\beta(K)\mid \alpha(K))$ is a reduction of type 1 of $(\beta(D)\mid \alpha(D))$.

\item Suppose that $D$ has a bigon in the crossing points $a_{n-1}$ and $a_{n}$ and let $(\Sigma,K)$ be the doodle diagram obtained by eliminating the bigon of $(\Sigma,D)$ at $a_{n-1}$ and $a_{n}$. Let us again denote by $\psi_K$, $\psi_{a_1}$,\ldots,$\psi_{a_{n-2}}$ the homology classes in $H_1(\sigma,\mathbf{Z})$ represented by the closed curves $K$, $K_{a_1}$,\ldots,$K_{a_{n-2}}$, respectively. Then, we have that $D_{a_{n-1}}+D_{a_{n}}$ is homologous to $D$, see Figure \ref{primibigon}.
\begin{figure}[ht]
    \centering
    \tikzset{every picture/.style={line width=0.75pt}} 

\begin{tikzpicture}[x=0.55pt,y=0.55pt,yscale=-1,xscale=1]
\draw    (162.8,171.58) .. controls (199.08,161) and (236.49,254.52) .. (186.54,220.73) ;
\draw [shift={(185,219.67)}, rotate = 35.12] [color={rgb, 255:red, 0; green, 0; blue, 0 }  ][line width=0.75]    (10.93,-3.29) .. controls (6.95,-1.4) and (3.31,-0.3) .. (0,0) .. controls (3.31,0.3) and (6.95,1.4) .. (10.93,3.29)   ;
\draw    (222.05,171.28) .. controls (148.8,178.24) and (177.68,222.69) .. (195.25,225.42) ;
\draw [shift={(224.3,171.08)}, rotate = 175.3] [color={rgb, 255:red, 0; green, 0; blue, 0 }  ][line width=0.75]    (10.93,-3.29) .. controls (6.95,-1.4) and (3.31,-0.3) .. (0,0) .. controls (3.31,0.3) and (6.95,1.4) .. (10.93,3.29)   ;

\draw    (148.2,38.47) .. controls (171.47,31.68) and (192.33,51) .. (194.11,65.64) .. controls (195.33,77.33) and (186.33,85.67) .. (187.37,92.55) .. controls (190.26,107.61) and (209.38,107.56) .. (231.42,106.16) ;
\draw [shift={(233.12,106.05)}, rotate = 176.23] [color={rgb, 255:red, 0; green, 0; blue, 0 }  ][line width=0.75]    (10.93,-3.29) .. controls (6.95,-1.4) and (3.31,-0.3) .. (0,0) .. controls (3.31,0.3) and (6.95,1.4) .. (10.93,3.29)   ;
\draw    (209.7,37.97) .. controls (176.92,40.66) and (165.32,54.86) .. (167.31,69.67) .. controls (168.41,77.89) and (179.33,84.33) .. (181.88,93.04) .. controls (182,102.97) and (168.14,105.15) .. (151.22,106.19) ;
\draw [shift={(149.37,106.3)}, rotate = 356.78] [color={rgb, 255:red, 0; green, 0; blue, 0 }  ][line width=0.75]    (10.93,-3.29) .. controls (6.95,-1.4) and (3.31,-0.3) .. (0,0) .. controls (3.31,0.3) and (6.95,1.4) .. (10.93,3.29)   ;
\draw    (32.2,37.55) .. controls (83.1,22.7) and (98.56,108.07) .. (34.82,106.86) ;
\draw [shift={(32.87,106.8)}, rotate = 2.6] [color={rgb, 255:red, 0; green, 0; blue, 0 }  ][line width=0.75]    (10.93,-3.29) .. controls (6.95,-1.4) and (3.31,-0.3) .. (0,0) .. controls (3.31,0.3) and (6.95,1.4) .. (10.93,3.29)   ;
\draw    (93.7,37.05) .. controls (18.38,43.24) and (54.87,110.19) .. (111.16,105.71) ;
\draw [shift={(112.87,105.55)}, rotate = 174.02] [color={rgb, 255:red, 0; green, 0; blue, 0 }  ][line width=0.75]    (10.93,-3.29) .. controls (6.95,-1.4) and (3.31,-0.3) .. (0,0) .. controls (3.31,0.3) and (6.95,1.4) .. (10.93,3.29)   ;
\draw    (342.12,39.63) .. controls (327.12,35.88) and (307.67,35) .. (308.78,44.63) .. controls (314.33,54.33) and (320.05,55.47) .. (319.66,65) .. controls (318.9,83.87) and (302.79,108.35) .. (273.43,108.02) ;
\draw [shift={(271.62,107.97)}, rotate = 2.6] [color={rgb, 255:red, 0; green, 0; blue, 0 }  ][line width=0.75]    (10.93,-3.29) .. controls (6.95,-1.4) and (3.31,-0.3) .. (0,0) .. controls (3.31,0.3) and (6.95,1.4) .. (10.93,3.29)   ;
\draw    (271.37,38.63) .. controls (286.53,33.63) and (301.67,35.67) .. (302.08,44) .. controls (301.33,51) and (288,59.05) .. (289.37,68.1) .. controls (292.25,87.15) and (317.48,107.04) .. (348.93,104.64) ;
\draw [shift={(350.87,104.47)}, rotate = 174.02] [color={rgb, 255:red, 0; green, 0; blue, 0 }  ][line width=0.75]    (10.93,-3.29) .. controls (6.95,-1.4) and (3.31,-0.3) .. (0,0) .. controls (3.31,0.3) and (6.95,1.4) .. (10.93,3.29)   ;

\draw    (35.8,172.33) .. controls (86.7,157.48) and (102.16,242.85) .. (38.42,241.65) ;
\draw [shift={(36.47,241.58)}, rotate = 2.6] [color={rgb, 255:red, 0; green, 0; blue, 0 }  ][line width=0.75]    (10.93,-3.29) .. controls (6.95,-1.4) and (3.31,-0.3) .. (0,0) .. controls (3.31,0.3) and (6.95,1.4) .. (10.93,3.29)   ;
\draw    (95.05,172.04) .. controls (21.98,179.44) and (59.79,246.27) .. (116.47,240.33) ;
\draw [shift={(97.3,171.83)}, rotate = 175.3] [color={rgb, 255:red, 0; green, 0; blue, 0 }  ][line width=0.75]    (10.93,-3.29) .. controls (6.95,-1.4) and (3.31,-0.3) .. (0,0) .. controls (3.31,0.3) and (6.95,1.4) .. (10.93,3.29)   ;
\draw    (322,194.17) .. controls (337.84,205.06) and (345.07,239.96) .. (281.17,237.75) ;
\draw [shift={(279.22,237.67)}, rotate = 2.6] [color={rgb, 255:red, 0; green, 0; blue, 0 }  ][line width=0.75]    (10.93,-3.29) .. controls (6.95,-1.4) and (3.31,-0.3) .. (0,0) .. controls (3.31,0.3) and (6.95,1.4) .. (10.93,3.29)   ;
\draw    (321.61,193.87) .. controls (286.34,167.14) and (289.86,247.52) .. (362.97,237.92) ;
\draw [shift={(323.25,195.17)}, rotate = 219.5] [color={rgb, 255:red, 0; green, 0; blue, 0 }  ][line width=0.75]    (10.93,-3.29) .. controls (6.95,-1.4) and (3.31,-0.3) .. (0,0) .. controls (3.31,0.3) and (6.95,1.4) .. (10.93,3.29)   ;
\draw    (151.54,239.61) .. controls (195.99,227.72) and (220.41,229.58) .. (247,240.42) ;
\draw [shift={(149.5,240.17)}, rotate = 344.72] [color={rgb, 255:red, 0; green, 0; blue, 0 }  ][line width=0.75]    (10.93,-3.29) .. controls (6.95,-1.4) and (3.31,-0.3) .. (0,0) .. controls (3.31,0.3) and (6.95,1.4) .. (10.93,3.29)   ;
\draw    (266,171.33) .. controls (287.25,166.83) and (292.58,182.01) .. (308.25,182.08) .. controls (323.69,182.16) and (321,165.83) .. (360.18,171.08) ;
\draw [shift={(362,171.33)}, rotate = 188.31] [color={rgb, 255:red, 0; green, 0; blue, 0 }  ][line width=0.75]    (10.93,-3.29) .. controls (6.95,-1.4) and (3.31,-0.3) .. (0,0) .. controls (3.31,0.3) and (6.95,1.4) .. (10.93,3.29)   ;

\draw (298.95,98.47) node [anchor=north west][inner sep=0.75pt]   [align=left] {$\displaystyle a_{n}$};
\draw (60.95,98.55) node [anchor=north west][inner sep=0.75pt]   [align=left] {$\displaystyle a_{n}$};
\draw (53.2,12.8) node [anchor=north west][inner sep=0.75pt]   [align=left] {$\displaystyle a_{n+1}$};
\draw (169.2,13.72) node [anchor=north west][inner sep=0.75pt]   [align=left] {$\displaystyle a_{n+1}$};
\draw (313.3,236.17) node [anchor=north west][inner sep=0.75pt]   [align=left] {$\displaystyle a_{n}$};
\draw (64.55,236.33) node [anchor=north west][inner sep=0.75pt]   [align=left] {$\displaystyle a_{n}$};
\draw (56.8,147.58) node [anchor=north west][inner sep=0.75pt]   [align=left] {$\displaystyle a_{n+1}$};
\draw (183.8,146.83) node [anchor=north west][inner sep=0.75pt]   [align=left] {$\displaystyle a_{n+1}$};

\draw (342.2,27.8) node [anchor=north west][inner sep=0.75pt]  [font=\scriptsize] [align=left] {$\displaystyle D_{a_{n+1}}$};
\draw (226.95,88.3) node [anchor=north west][inner sep=0.75pt]  [font=\scriptsize] [align=left] {$\displaystyle D_{a_{n}}$};

\draw (336.3,150.58) node [anchor=north west][inner sep=0.75pt]  [font=\scriptsize] [align=left] {$\displaystyle D_{a_{n+1}}$};
\draw (218.3,240.58) node [anchor=north west][inner sep=0.75pt]  [font=\scriptsize] [align=left] {$\displaystyle D_{a_{n}}$};

\end{tikzpicture}
    \caption{Generating primitive curves in a bigon}
    \label{primibigon}
\end{figure}
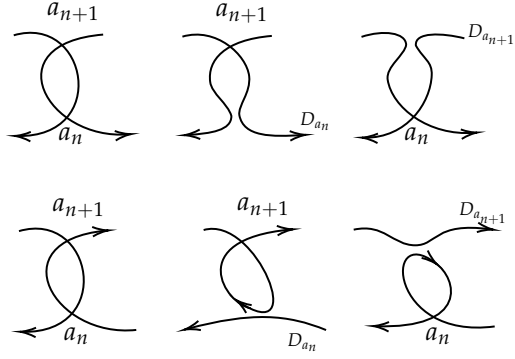

In this way, $\varphi_{a_{n-1}}+\varphi_{a_n}=\varphi_{D}$. Therefore, $ \left\langle \varphi_{a_i},\varphi_{D} \right\rangle=  \left\langle \varphi_{a_i},\varphi_{a_{n-1}} \right\rangle+ \left\langle \varphi_{a_i},\varphi_{a_n} \right\rangle,$ for all $i=1,\ldots,n.$ Thus, if $i=n$, $ \left\langle \varphi_{a_n},\varphi_{D} \right\rangle=  \left\langle \varphi_{a_n},\varphi_{a_{n-1}} \right\rangle$ and if $i=n-1$, $ \left\langle \varphi_{a_{n-1}},\varphi_{D} \right\rangle=  \left\langle \varphi_{a_{n-1}},\varphi_{a_n} \right\rangle$. As a consequence, $ \left\langle \varphi_{a_n},\varphi_{D} \right\rangle=  -\left\langle \varphi_{a_{n-1}},\varphi_{D} \right\rangle.$ On the other hand, if $i\notin \{n-1,n\}$, we have the following equivalence
\[
a_{k}^{\epsilon}\in D_{a_i}\nearrow D \iff a_{k}^{-\epsilon}\in D_{a_i}\nearrow D, \text{ for all } k\in \{n-1,n\}.
\]
Hence, $\left\langle \varphi_{a_i},\varphi_{D} \right\rangle=  \left\langle \psi_{a_{i}},\psi_{K} \right\rangle.$ Moreover, for $j\notin \{n-1,n\}$, we also have the following equivalence,
\[
a_{k}^{\epsilon}\in D_{a_i}\nearrow D_{a_j} \iff a_{k}^{-\epsilon}\in D_{a_i}\nearrow D_{a_j}, \text{ for all } k\in \{n-1,n\} .
\]
This proves that $\left\langle \varphi_{a_i},\varphi_{a_j} \right\rangle=  \left\langle \psi_{a_{i}},\psi_{a_j} \right\rangle.$ Thereby, $(\beta(D)\mid \alpha(D))$ has the form:

\begin{equation*}
    (\beta(D)\mid \alpha(D))=\left( 
	\begin{array}[c]{ccc|c}
	\beta(K) &A &B& \alpha(K)\\
 -A^{T}&0&\alpha_{n-1}(D)&\alpha_{n-1}(D)\\
 -B^{T}&-\alpha_{n-1}(D)&0&-\alpha_{n-1}(D)
 
	\end{array}%
 \right),
\end{equation*}
where $A+B=\alpha(K)$. Therefore, $(\beta(K)\mid \alpha(K))$ is a reduction of type 2 of $(\beta(D)\mid \alpha(D))$.
\end{enumerate} 
\end{proof}

    \begin{definition}
	Let $(B \mid A)$ be a skew-symmetric augmented matrix of dimension $n$. We say that $(B \mid A)$ is \textbf{trivial} if it is equivalent to the empty augmented matrix. $(B \mid A)$ is \textbf{reducible} if a reduction transformation can be applied to it. Otherwise, we say $(B \mid A)$ is \textbf{irreducible}.
\end{definition}

\begin{theorem}
	Let $\left[B\mid A\right] \in \mathbf{Skew}/_{\asymp}$. Then there exists $(B^{\prime}\mid A^{\prime})\in\left[ B\mid A\right]  $, such that $(B^{\prime}\mid A^{\prime})$ is an irreducible skew-symmetric augmented matrix. Moreover, if $(B^{\prime\prime}\mid A^{\prime\prime})$ is another representative irreducible in $\left[B\mid A\right]  $, then there exists a permutation matrix $P$ such that $(PB^{\prime}P^{T}\mid PA^{\prime})=(B^{\prime\prime}\mid A^{\prime\prime})$, so all irreducible representatives  of $\left[B\mid A\right]  $ have the same dimension.
\end{theorem}
\begin{proof}
  Let $\left(B\mid A\right) \in \mathbf{Skew}$ be a skew-symmetric augmented matrix and let 
   \[\mathcal{F}=\{m\in \mathbb{N} \cup \{0\} \mid \exists (B'\mid A')\in\left[B\mid A\right] \text{s. t. $m$ is the order of  }(B'\mid A')   \}.
   \] 
Because $\mathcal{F} \subset \mathbb{N}\cup \{0\}$ and $\mathcal{F} \neq \emptyset$, then from the well-ordering principle, there exists a minimal element $n_{0}\in \mathcal{F}$ which corresponds to a skew-symmetric augmented matrix $(B'\mid A')\in\left[B\mid A\right]$. Due to the minimality of $n_0$, $(B'\mid A')$ has to be irreducible. Besides, let $(B''\mid A'')$ be another irreducible in $\left[B\mid A\right]$ of dimension $m$. Because, $m\in \mathcal{F}$, then $n_{0}\leq m$. Since, $(B''\mid A'')$ and $(B'\mid A')$ are $S$-equivalent, then there exists a finite collection $\{(B_{i}\mid A_{i})\}_{i=0}^{n}$ of skew-symmetric augmented matrices such that 
\begin{equation}
(B''\mid A'')=(B_{0}\mid A_{0}) \asymp (B_{1}\mid A_{1})\asymp (B_{2}\mid A_{2})\asymp \cdots \asymp (B_{n-1}\mid A_{n-1})\asymp (B_{n}\mid A_{n})=(B'\mid A'),
    \label{coll01}
\end{equation}
where $(B_{i+1}\mid A_{i+1})$ is obtained from $(B_i\mid A_i)$ by using only a permutation or an elementary transformation. We denote the dimension of  $(B_{i}\mid A_{i})$ by $m_i$, and we prove that, for every $i=0,\ldots,n$, $m\leq m_i$ and if $m=m_i$ then $(B_{i}\mid A_{i})$ is a permutation of $(B''\mid A'')$. Let us proceed by induction on the length $n$ of the collections of the form given in Equation (\ref{coll01}). If $n=1$, then $(B'\mid A')$ has to be a permutation of $(B''\mid A'')$, thus $m\leq m_1=n_0$. Suppose that the sentence is true for collections of length $k$ and suppose that we have a collection $\{(B_{i}\mid A_{i})\}_{i=0}^{k+1}$ that satisfies (\ref{coll01}).  Hence, $n_{i}\geq m$ and, if $n_i=m$, then $(B_i\mid A_)$ is a permutation of $(B''\mid A'')$, for every $i=1,2,\ldots,k$. If $n_{k}=m$, then  $(B_k\mid A_k)$ is a permutation of $(B''\mid A'')$. So, $(B_{k+1}\mid A_{k+1})$ has to be either a permutation or an extension of $(B_k\mid A_k)$. In any case, $m\leq n_{k+1}$ and if $m= n_{k+1}$, we have that  $(B_{k+1}\mid A_{k+1})$ is a permutation of $ (B''\mid A'')$ . On the other hand,  if $n_{k}>m$, we have either $n_{k+1}=n_k$ or $n_{k+1}>n_k$ or $n_{k+1}=n_k - 1$, thus $n_{k+1} \geq m$. In this way, if $n_{k+1}=m$, then $(B_{k+1}\mid A_{k+1})$ is a simplification of $(B_{k}\mid A_{k})$. Therefore, we have two cases: In the first case, $n_{k}=m+1$ then there exists $n_{j}$ such that $n_{j}=m$ and $(B_{j+1}\mid A_{j+1})\asymp (B_{j+2}\mid A_{j+2}) \asymp \cdots \asymp (B_{k}\mid A_{k})$ is a sequence which only uses permutations and $(B_{j+1}\mid A_{j+1})$ is an extension of $(B_{j}\mid A_{j})$ of type $1$. Thus, the simplification used to change $(B_{k}\mid A_{})$ into $(B_{k+1},A_{k+1})$ is the same used to change $(B_{j}\mid A_{j})$ into $(B_{j+1}\mid A_{j+1})$. Therefore,  $(B_{k+1}\mid A_{k+1})$ is a permutation of $(B_{j}\mid A_{j})$ which is a permutation of $(B''\mid A'')$, so  $(B_{k+1}\mid A_{k+1})$ is a permutation  $(B''\mid A'')$. The second case, $n_{k}=m+2$, comes in the same fashion as the first one. As a consequence, $m\leq m_{n}=n_0$, and therefore $m=n_0$. From the previous analysis,  $(B'\mid A')$ is a permutation of  $(B''\mid A'')$.

\end{proof}

    \begin{corollary}
	If $(\Sigma,D)$ is a classical doodle, then $\Lambda([w])=[\ \ \mid \ \ ]$
\end{corollary}

The proof of the following lemma is a direct consequence of the definition and  properties of the homology intersection number, therefore we omit it.

\begin{lemma}
    Let $(\Sigma,D)$ be a one-component doodle diagram with crossing points labeled with the letters $a_1,\ldots,a_n$. Let $\varphi_{D},\varphi_{a_1},\ldots,\varphi_{a_n}$ and $\widetilde{\varphi}_{a_1},\ldots,\widetilde{\varphi}_{a_n}$ denote the equivalence classes represented by the curves $D,D_{a_1},\ldots,D_{a_n}$ and $\widetilde{D}_{a_1},\ldots,\widetilde{D}_{a_n}$, respectively. Then,
    \begin{equation}
\begin{array}{ccl}
     \left\langle \varphi_{a_i},\varphi_{D} \right\rangle&=&\left\langle \varphi_{a_i},\widetilde{\varphi}_{a_i} \right\rangle=-\left\langle \widetilde{\varphi}_{a_i},\varphi_{D} \right\rangle \text{ and }\\
     &&\\
     \left\langle \varphi_{a_i},\varphi_{a_j} \right\rangle&=&\left\langle \widetilde{\varphi}_{a_i},\widetilde{\varphi}_{a_j} \right\rangle +\left\langle \varphi_{a_i},\varphi_{D} \right\rangle-\left\langle \varphi_{a_j},\varphi_{D} \right\rangle,
\end{array}
\label{eq:rela}
    \end{equation}
    for every $i,j$.
\end{lemma}

\begin{example}
The Kishino's knot is an important example in virtual knot theory because it cannot be differentiated from  the trivial knot by the fundamental group and the bracket polynomial. The Kishino's knot was distinguished by 3-strand Jones polynomial, the surface bracket polynomial, the quaternionic biquandle and by skew-symmetric graded matrices, see \cite{Ro2}. In this paper we use skew-symmetric augmented matrices to distinguish the Kishino's doodle or the flat Kishino knot, see Figure~\ref{kishino}. The Kishino's doodle has been distinguished from the trivial doodle in \cite[Page~15]{FeTu} by using representations of Weyl algebras and in  \cite[Theorem~4.1]{Ka}. 

 \begin{figure}[ht]
	\begin{center}
		\tikzset{every picture/.style={line width=0.75pt}} 

\begin{tikzpicture}[x=0.75pt,y=0.75pt,yscale=-1,xscale=1]

\draw  [line width=2.25]  (59.2,116.3) .. controls (59.2,81.45) and (127.07,53.2) .. (210.8,53.2) .. controls (294.53,53.2) and (362.4,81.45) .. (362.4,116.3) .. controls (362.4,151.15) and (294.53,179.4) .. (210.8,179.4) .. controls (127.07,179.4) and (59.2,151.15) .. (59.2,116.3) -- cycle ;
\draw  [line width=2.25]  (112.8,117.3) .. controls (112.8,110.29) and (122.65,104.6) .. (134.8,104.6) .. controls (146.95,104.6) and (156.8,110.29) .. (156.8,117.3) .. controls (156.8,124.31) and (146.95,130) .. (134.8,130) .. controls (122.65,130) and (112.8,124.31) .. (112.8,117.3) -- cycle ;
\draw  [line width=2.25]  (249.6,117.3) .. controls (249.6,110.29) and (259.45,104.6) .. (271.6,104.6) .. controls (283.75,104.6) and (293.6,110.29) .. (293.6,117.3) .. controls (293.6,124.31) and (283.75,130) .. (271.6,130) .. controls (259.45,130) and (249.6,124.31) .. (249.6,117.3) -- cycle ;
\draw [line width=1.5]    (107.2,71) .. controls (136.98,64.43) and (173.65,78.35) .. (180.62,113.88) .. controls (187.6,149.4) and (112.8,189) .. (93.98,126.4) .. controls (75.54,65.06) and (247.31,71.51) .. (283.18,105.3) ;
\draw [shift={(285.2,107.4)}, rotate = 229.13] [color={rgb, 255:red, 0; green, 0; blue, 0 }  ][line width=1.5]    (14.21,-6.37) .. controls (9.04,-2.99) and (4.3,-0.87) .. (0,0) .. controls (4.3,0.87) and (9.04,2.99) .. (14.21,6.37)   ;
\draw [line width=1.5]  [dash pattern={on 1.69pt off 2.76pt}]  (120,107.8) .. controls (103.2,112.2) and (67.2,101) .. (107.2,71) ;
\draw [line width=1.5]    (116,109.4) .. controls (69.5,131.37) and (143.6,160.6) .. (190,159.4) .. controls (236.4,158.2) and (273.6,153) .. (308.4,132.2) .. controls (343.2,111.4) and (302.38,68.97) .. (259.99,75.73) .. controls (217.6,82.49) and (221.86,176.37) .. (303.33,153.29) .. controls (384.8,130.2) and (346.8,95) .. (333.2,78.6) ;
\draw [line width=1.5]  [dash pattern={on 1.69pt off 2.76pt}]  (285.2,107.4) .. controls (318.31,119.53) and (305.64,69.53) .. (335.64,80.2) ;
\draw [line width=1.5]    (188.95,80.03) -- (201.7,80.64) ;
\draw [shift={(204.7,80.78)}, rotate = 182.73] [color={rgb, 255:red, 0; green, 0; blue, 0 }  ][line width=1.5]    (14.21,-4.28) .. controls (9.04,-1.82) and (4.3,-0.39) .. (0,0) .. controls (4.3,0.39) and (9.04,1.82) .. (14.21,4.28)   ;
\draw [line width=1.5]    (233.87,115.03) -- (234.1,107.78) ;
\draw [shift={(234.2,104.78)}, rotate = 91.86] [color={rgb, 255:red, 0; green, 0; blue, 0 }  ][line width=1.5]    (14.21,-4.28) .. controls (9.04,-1.82) and (4.3,-0.39) .. (0,0) .. controls (4.3,0.39) and (9.04,1.82) .. (14.21,4.28)   ;
\draw [line width=1.5]    (225.45,157.2) -- (207.6,158.98) ;
\draw [shift={(204.62,159.28)}, rotate = 354.29] [color={rgb, 255:red, 0; green, 0; blue, 0 }  ][line width=1.5]    (14.21,-4.28) .. controls (9.04,-1.82) and (4.3,-0.39) .. (0,0) .. controls (4.3,0.39) and (9.04,1.82) .. (14.21,4.28)   ;

\draw (247.7,132.28) node [anchor=north west][inner sep=0.75pt]   [align=left] {$\displaystyle a_{4}$};
\draw (224.95,69.53) node [anchor=north west][inner sep=0.75pt]   [align=left] {$\displaystyle a_{3}$};
\draw (144.95,132.78) node [anchor=north west][inner sep=0.75pt]   [align=left] {$\displaystyle a_{2}$};
\draw (153.7,66.53) node [anchor=north west][inner sep=0.75pt]   [align=left] {$\displaystyle a_{1}$};

\end{tikzpicture}

	\end{center}
 \caption{The Kishino's doodle}
 \label{kishino}
\end{figure}
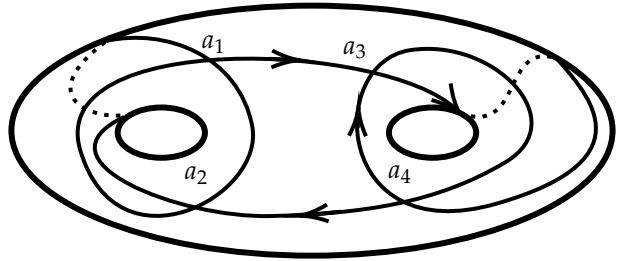 

Let us denote the Kishino's doodle diagram by $(T,D)$, where $T$ is the torus surface. In Figure \ref{kishino2} we show the construction of the generating set of primitive curves of  the first homology group $H_{1}(T,\mathbf{Z})$ of $T$. 
 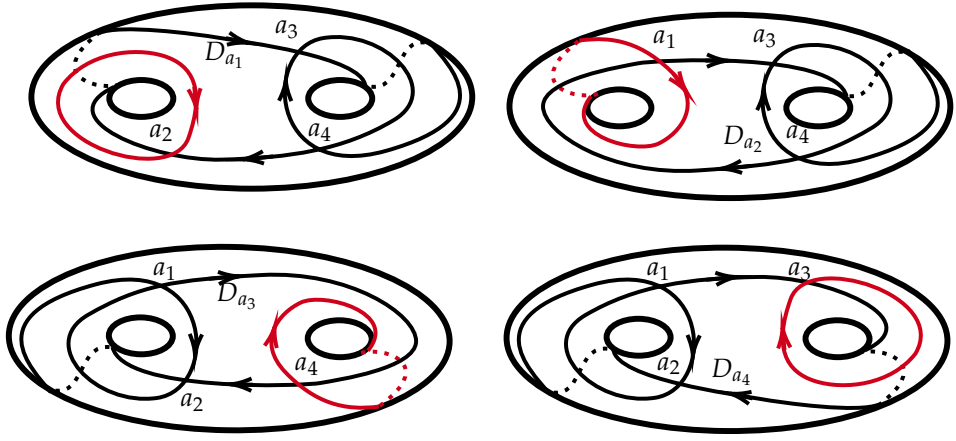
\begin{figure}[ht]
	\begin{center}
		\tikzset{every picture/.style={line width=0.55pt}} 

\begin{tikzpicture}[x=0.55pt,y=0.55pt,yscale=-1,xscale=1]
\draw  [line width=2.25]  (651,476.44) .. controls (650.73,511.29) and (582.65,539.03) .. (498.93,538.4) .. controls (415.2,537.77) and (347.54,509.01) .. (347.8,474.16) .. controls (348.07,439.31) and (416.15,411.57) .. (499.87,412.2) .. controls (583.6,412.83) and (651.26,441.59) .. (651,476.44) -- cycle ;
\draw  [line width=2.25]  (597.4,475.04) .. controls (597.35,482.05) and (587.46,487.66) .. (575.31,487.57) .. controls (563.16,487.48) and (553.35,481.72) .. (553.41,474.71) .. controls (553.46,467.69) and (563.35,462.08) .. (575.5,462.17) .. controls (587.65,462.26) and (597.46,468.02) .. (597.4,475.04) -- cycle ;
\draw  [line width=2.25]  (460.61,474.01) .. controls (460.56,481.02) and (450.66,486.63) .. (438.51,486.54) .. controls (426.36,486.45) and (416.56,480.69) .. (416.61,473.68) .. controls (416.66,466.66) and (426.55,461.05) .. (438.7,461.14) .. controls (450.85,461.24) and (460.66,467) .. (460.61,474.01) -- cycle ;
\draw [line width=1.5]  [dash pattern={on 1.69pt off 2.76pt}]  (590.13,484.48) .. controls (606.97,480.21) and (642.88,491.68) .. (602.66,521.38) ;
\draw [line width=1.5]    (594.15,482.91) .. controls (640.81,461.29) and (566.93,431.51) .. (520.52,432.36) .. controls (474.12,433.21) and (436.88,438.13) .. (401.92,458.67) .. controls (366.97,479.21) and (407.47,521.94) .. (449.91,515.5) .. controls (492.35,509.06) and (488.79,415.15) .. (407.15,437.62) .. controls (325.51,460.09) and (363.24,495.58) .. (376.72,512.08) ;
\draw [line width=1.5]  [dash pattern={on 1.69pt off 2.76pt}]  (424.94,483.64) .. controls (391.92,471.26) and (404.21,521.35) .. (374.29,510.46) ;
\draw [line width=1.5]    (516.95,515.95) -- (504.39,513.18) ;
\draw [shift={(501.46,512.53)}, rotate = 12.45] [color={rgb, 255:red, 0; green, 0; blue, 0 }  ][line width=1.5]    (14.21,-4.28) .. controls (9.04,-1.82) and (4.3,-0.39) .. (0,0) .. controls (4.3,0.39) and (9.04,1.82) .. (14.21,4.28)   ;
\draw [line width=1.5]    (476.32,476.39) -- (476.03,483.64) ;
\draw [shift={(475.91,486.64)}, rotate = 272.29] [color={rgb, 255:red, 0; green, 0; blue, 0 }  ][line width=1.5]    (14.21,-4.28) .. controls (9.04,-1.82) and (4.3,-0.39) .. (0,0) .. controls (4.3,0.39) and (9.04,1.82) .. (14.21,4.28)   ;
\draw [line width=1.5]    (485.06,434.29) -- (502.92,432.64) ;
\draw [shift={(505.91,432.36)}, rotate = 174.72] [color={rgb, 255:red, 0; green, 0; blue, 0 }  ][line width=1.5]    (14.21,-4.28) .. controls (9.04,-1.82) and (4.3,-0.39) .. (0,0) .. controls (4.3,0.39) and (9.04,1.82) .. (14.21,4.28)   ;
\draw  [line width=2.25]  (308.63,476.96) .. controls (308.13,511.81) and (239.85,539.07) .. (156.14,537.86) .. controls (72.42,536.65) and (4.96,507.42) .. (5.47,472.57) .. controls (5.97,437.73) and (74.25,410.46) .. (157.96,411.67) .. controls (241.68,412.89) and (309.14,442.12) .. (308.63,476.96) -- cycle ;
\draw  [line width=2.25]  (255.05,475.19) .. controls (254.95,482.2) and (245.02,487.74) .. (232.87,487.57) .. controls (220.72,487.39) and (210.96,481.56) .. (211.06,474.55) .. controls (211.16,467.54) and (221.09,461.99) .. (233.24,462.17) .. controls (245.39,462.35) and (255.16,468.17) .. (255.05,475.19) -- cycle ;
\draw  [line width=2.25]  (118.27,473.2) .. controls (118.17,480.22) and (108.24,485.76) .. (96.09,485.58) .. controls (83.94,485.41) and (74.17,479.58) .. (74.27,472.57) .. controls (74.37,465.55) and (84.31,460.01) .. (96.45,460.19) .. controls (108.6,460.36) and (118.37,466.19) .. (118.27,473.2) -- cycle ;
\draw [color={rgb, 255:red, 208; green, 2; blue, 27 }  ,draw opacity=1 ][line width=1.5]  [dash pattern={on 1.69pt off 2.76pt}]  (247.72,484.58) .. controls (264.58,480.42) and (300.41,492.15) .. (259.98,521.56) ;
\draw [line width=1.5]    (76.75,480.53) .. controls (84.04,511.89) and (207.62,507.52) .. (219.75,505.34) .. controls (231.89,503.15) and (285.29,495.39) .. (285.57,476.72) .. controls (285.85,458.06) and (239.53,427.89) .. (179.48,431.02) .. controls (119.42,434.15) and (95.02,438.16) .. (59.93,458.46) .. controls (24.83,478.75) and (65.03,521.77) .. (107.51,515.62) .. controls (150,509.48) and (147.1,415.54) .. (65.3,437.45) .. controls (-16.49,459.35) and (20.99,495.1) .. (34.35,511.69) ;
\draw [line width=1.5]  [dash pattern={on 1.69pt off 2.76pt}]  (82.53,482.59) .. controls (49.6,469.98) and (61.54,520.15) .. (31.7,509.05) ;
\draw [line width=1.5]    (173.44,506.52) -- (161.13,506.28) ;
\draw [shift={(158.13,506.22)}, rotate = 1.15] [color={rgb, 255:red, 0; green, 0; blue, 0 }  ][line width=1.5]    (14.21,-4.28) .. controls (9.04,-1.82) and (4.3,-0.39) .. (0,0) .. controls (4.3,0.39) and (9.04,1.82) .. (14.21,4.28)   ;
\draw [line width=1.5]    (133.97,475.7) -- (133.63,482.95) ;
\draw [shift={(133.49,485.94)}, rotate = 272.69] [color={rgb, 255:red, 0; green, 0; blue, 0 }  ][line width=1.5]    (14.21,-4.28) .. controls (9.04,-1.82) and (4.3,-0.39) .. (0,0) .. controls (4.3,0.39) and (9.04,1.82) .. (14.21,4.28)   ;
\draw [line width=1.5]    (142.99,433.66) -- (160.87,432.13) ;
\draw [shift={(163.86,431.88)}, rotate = 175.12] [color={rgb, 255:red, 0; green, 0; blue, 0 }  ][line width=1.5]    (14.21,-4.28) .. controls (9.04,-1.82) and (4.3,-0.39) .. (0,0) .. controls (4.3,0.39) and (9.04,1.82) .. (14.21,4.28)   ;
\draw [line width=1.5]    (604,520.1) .. controls (590.04,526.25) and (538.13,519.11) .. (521.86,516.24) .. controls (505.59,513.36) and (422.91,502.56) .. (421.28,482.23) ;
\draw [color={rgb, 255:red, 208; green, 2; blue, 27 }  ,draw opacity=1 ][line width=1.5]    (628.93,485.21) .. controls (612.27,506.58) and (580.17,510.47) .. (559.87,504.02) .. controls (539.57,497.58) and (531.22,480.61) .. (543.21,447.39) .. controls (555.2,414.18) and (652.91,445.19) .. (628.93,485.21) -- cycle ;
\draw [color={rgb, 255:red, 208; green, 2; blue, 27 }  ,draw opacity=1 ][line width=1.5]    (538.95,482.69) -- (537.82,470.04) ;
\draw [shift={(537.55,467.05)}, rotate = 84.91] [color={rgb, 255:red, 208; green, 2; blue, 27 }  ,draw opacity=1 ][line width=1.5]    (14.21,-4.28) .. controls (9.04,-1.82) and (4.3,-0.39) .. (0,0) .. controls (4.3,0.39) and (9.04,1.82) .. (14.21,4.28)   ;
\draw [color={rgb, 255:red, 208; green, 2; blue, 27 }  ,draw opacity=1 ][line width=1.5]    (259.98,521.56) .. controls (236.12,524.67) and (216.08,519.15) .. (195.76,498.01) .. controls (175.43,476.87) and (190.95,457.09) .. (206.46,449.91) .. controls (221.96,442.73) and (274.09,456.6) .. (252.48,483.08) ;
\draw  [line width=2.25]  (349.8,314.97) .. controls (349.8,280.12) and (417.67,251.87) .. (501.4,251.87) .. controls (585.13,251.87) and (653,280.12) .. (653,314.97) .. controls (653,349.82) and (585.13,378.07) .. (501.4,378.07) .. controls (417.67,378.07) and (349.8,349.82) .. (349.8,314.97) -- cycle ;
\draw  [line width=2.25]  (403.4,315.97) .. controls (403.4,308.95) and (413.25,303.27) .. (425.4,303.27) .. controls (437.55,303.27) and (447.4,308.95) .. (447.4,315.97) .. controls (447.4,322.98) and (437.55,328.67) .. (425.4,328.67) .. controls (413.25,328.67) and (403.4,322.98) .. (403.4,315.97) -- cycle ;
\draw  [line width=2.25]  (540.2,315.97) .. controls (540.2,308.95) and (550.05,303.27) .. (562.2,303.27) .. controls (574.35,303.27) and (584.2,308.95) .. (584.2,315.97) .. controls (584.2,322.98) and (574.35,328.67) .. (562.2,328.67) .. controls (550.05,328.67) and (540.2,322.98) .. (540.2,315.97) -- cycle ;
\draw [color={rgb, 255:red, 208; green, 2; blue, 27 }  ,draw opacity=1 ][line width=1.5]  [dash pattern={on 1.69pt off 2.76pt}]  (410.6,306.47) .. controls (393.8,310.87) and (357.8,299.67) .. (397.8,269.67) ;
\draw [line width=1.5]    (581.61,308.04) .. controls (573.86,276.79) and (450.36,282.95) .. (438.26,285.31) .. controls (426.16,287.67) and (372.88,296.21) .. (372.87,314.87) .. controls (372.86,333.54) and (419.61,363.04) .. (479.61,359.04) .. controls (539.61,355.04) and (563.95,350.67) .. (598.75,329.87) .. controls (633.55,309.07) and (592.73,266.64) .. (550.34,273.4) .. controls (507.95,280.15) and (512.21,374.04) .. (593.68,350.95) .. controls (675.15,327.87) and (637.15,292.67) .. (623.55,276.27) ;
\draw [line width=1.5]  [dash pattern={on 1.69pt off 2.76pt}]  (575.8,306.07) .. controls (608.91,318.2) and (596.24,268.2) .. (626.24,278.87) ;
\draw [line width=1.5]    (484.55,283.45) -- (496.86,283.52) ;
\draw [shift={(499.86,283.54)}, rotate = 180.32] [color={rgb, 255:red, 0; green, 0; blue, 0 }  ][line width=1.5]    (14.21,-4.28) .. controls (9.04,-1.82) and (4.3,-0.39) .. (0,0) .. controls (4.3,0.39) and (9.04,1.82) .. (14.21,4.28)   ;
\draw [line width=1.5]    (524.47,313.7) -- (524.7,306.45) ;
\draw [shift={(524.8,303.45)}, rotate = 91.86] [color={rgb, 255:red, 0; green, 0; blue, 0 }  ][line width=1.5]    (14.21,-4.28) .. controls (9.04,-1.82) and (4.3,-0.39) .. (0,0) .. controls (4.3,0.39) and (9.04,1.82) .. (14.21,4.28)   ;
\draw [line width=1.5]    (516.05,355.87) -- (498.2,357.65) ;
\draw [shift={(495.22,357.95)}, rotate = 354.29] [color={rgb, 255:red, 0; green, 0; blue, 0 }  ][line width=1.5]    (14.21,-4.28) .. controls (9.04,-1.82) and (4.3,-0.39) .. (0,0) .. controls (4.3,0.39) and (9.04,1.82) .. (14.21,4.28)   ;
\draw [color={rgb, 255:red, 208; green, 2; blue, 27 }  ,draw opacity=1 ][line width=1.5]    (397.8,269.67) .. controls (421.62,266.21) and (441.73,271.45) .. (462.36,292.29) .. controls (482.99,313.13) and (467.76,333.14) .. (452.36,340.54) .. controls (436.96,347.94) and (384.64,334.83) .. (405.86,308.04) ;
\draw  [line width=2.25]  (21,308.5) .. controls (21,273.65) and (88.87,245.4) .. (172.6,245.4) .. controls (256.33,245.4) and (324.2,273.65) .. (324.2,308.5) .. controls (324.2,343.35) and (256.33,371.6) .. (172.6,371.6) .. controls (88.87,371.6) and (21,343.35) .. (21,308.5) -- cycle ;
\draw  [line width=2.25]  (74.6,309.5) .. controls (74.6,302.49) and (84.45,296.8) .. (96.6,296.8) .. controls (108.75,296.8) and (118.6,302.49) .. (118.6,309.5) .. controls (118.6,316.51) and (108.75,322.2) .. (96.6,322.2) .. controls (84.45,322.2) and (74.6,316.51) .. (74.6,309.5) -- cycle ;
\draw  [line width=2.25]  (211.4,309.5) .. controls (211.4,302.49) and (221.25,296.8) .. (233.4,296.8) .. controls (245.55,296.8) and (255.4,302.49) .. (255.4,309.5) .. controls (255.4,316.51) and (245.55,322.2) .. (233.4,322.2) .. controls (221.25,322.2) and (211.4,316.51) .. (211.4,309.5) -- cycle ;
\draw [line width=1.5]  [dash pattern={on 1.69pt off 2.76pt}]  (81.8,300) .. controls (65,304.4) and (29,293.2) .. (69,263.2) ;
\draw [line width=1.5]    (77.8,301.6) .. controls (31.3,323.57) and (105.4,352.8) .. (151.8,351.6) .. controls (198.2,350.4) and (235.4,345.2) .. (270.2,324.4) .. controls (305,303.6) and (264.18,261.17) .. (221.79,267.93) .. controls (179.4,274.69) and (183.66,368.57) .. (265.13,345.49) .. controls (346.6,322.4) and (308.6,287.2) .. (295,270.8) ;
\draw [line width=1.5]  [dash pattern={on 1.69pt off 2.76pt}]  (247,299.6) .. controls (280.11,311.73) and (267.44,261.73) .. (297.44,272.4) ;
\draw [line width=1.5]    (154.75,267.98) -- (167.33,270.66) ;
\draw [shift={(170.26,271.29)}, rotate = 192.02] [color={rgb, 255:red, 0; green, 0; blue, 0 }  ][line width=1.5]    (14.21,-4.28) .. controls (9.04,-1.82) and (4.3,-0.39) .. (0,0) .. controls (4.3,0.39) and (9.04,1.82) .. (14.21,4.28)   ;
\draw [line width=1.5]    (195.67,307.23) -- (195.9,299.98) ;
\draw [shift={(196,296.98)}, rotate = 91.86] [color={rgb, 255:red, 0; green, 0; blue, 0 }  ][line width=1.5]    (14.21,-4.28) .. controls (9.04,-1.82) and (4.3,-0.39) .. (0,0) .. controls (4.3,0.39) and (9.04,1.82) .. (14.21,4.28)   ;
\draw [line width=1.5]    (187.25,349.4) -- (169.4,351.18) ;
\draw [shift={(166.42,351.48)}, rotate = 354.29] [color={rgb, 255:red, 0; green, 0; blue, 0 }  ][line width=1.5]    (14.21,-4.28) .. controls (9.04,-1.82) and (4.3,-0.39) .. (0,0) .. controls (4.3,0.39) and (9.04,1.82) .. (14.21,4.28)   ;
\draw [line width=1.5]    (67.67,264.48) .. controls (81.58,258.23) and (133.54,264.98) .. (149.83,267.73) .. controls (166.12,270.49) and (248.88,280.66) .. (250.67,300.98) ;
\draw [color={rgb, 255:red, 208; green, 2; blue, 27 }  ,draw opacity=1 ][line width=1.5]    (43,299.57) .. controls (59.5,278.07) and (91.57,273.94) .. (111.92,280.23) .. controls (132.27,286.53) and (140.74,303.43) .. (129,336.73) .. controls (117.26,370.04) and (19.32,339.76) .. (43,299.57) -- cycle ;
\draw [color={rgb, 255:red, 208; green, 2; blue, 27 }  ,draw opacity=1 ][line width=1.5]    (133,301.41) -- (134.22,314.05) ;
\draw [shift={(134.51,317.04)}, rotate = 264.48] [color={rgb, 255:red, 208; green, 2; blue, 27 }  ,draw opacity=1 ][line width=1.5]    (14.21,-4.28) .. controls (9.04,-1.82) and (4.3,-0.39) .. (0,0) .. controls (4.3,0.39) and (9.04,1.82) .. (14.21,4.28)   ;
\draw [color={rgb, 255:red, 208; green, 2; blue, 27 }  ,draw opacity=1 ][line width=1.5]    (462.36,292.29) -- (469.75,302.46) ;
\draw [shift={(471.51,304.89)}, rotate = 234.01] [color={rgb, 255:red, 208; green, 2; blue, 27 }  ,draw opacity=1 ][line width=1.5]    (14.21,-4.28) .. controls (9.04,-1.82) and (4.3,-0.39) .. (0,0) .. controls (4.3,0.39) and (9.04,1.82) .. (14.21,4.28)   ;
\draw [color={rgb, 255:red, 208; green, 2; blue, 27 }  ,draw opacity=1 ][line width=1.5]    (186.13,477.76) -- (187.8,469.45) ;
\draw [shift={(188.38,466.51)}, rotate = 101.32] [color={rgb, 255:red, 208; green, 2; blue, 27 }  ,draw opacity=1 ][line width=1.5]    (14.21,-4.28) .. controls (9.04,-1.82) and (4.3,-0.39) .. (0,0) .. controls (4.3,0.39) and (9.04,1.82) .. (14.21,4.28)   ;
\draw (146.4,434.7) node [anchor=north west][inner sep=0.75pt]   [align=left] {$\displaystyle D_{a_{3}}$};
\draw (102.85,420.53) node [anchor=north west][inner sep=0.75pt]   [align=left] {$\displaystyle a_{1}$};
\draw (121.6,510.6) node [anchor=north west][inner sep=0.75pt]   [align=left] {$\displaystyle a_{2}$};
\draw (200.35,485.55) node [anchor=north west][inner sep=0.75pt]   [align=left] {$\displaystyle a_{4}$};
\draw (486.55,486.7) node [anchor=north west][inner sep=0.75pt]   [align=left] {$\displaystyle D_{a_{4}}$};
\draw (442.4,420.33) node [anchor=north west][inner sep=0.75pt]   [align=left] {$\displaystyle a_{1}$};
\draw (539.55,420.6) node [anchor=north west][inner sep=0.75pt]   [align=left] {$\displaystyle a_{3}$};
\draw (449.5,485.55) node [anchor=north west][inner sep=0.75pt]   [align=left] {$\displaystyle a_{2}$};
\draw (494.55,326.7) node [anchor=north west][inner sep=0.75pt]   [align=left] {$\displaystyle D_{a_{2}}$};
\draw (446.8,260.53) node [anchor=north west][inner sep=0.75pt]   [align=left] {$\displaystyle a_{1}$};
\draw (515.55,260.2) node [anchor=north west][inner sep=0.75pt]   [align=left] {$\displaystyle a_{3}$};
\draw (538.3,325.95) node [anchor=north west][inner sep=0.75pt]   [align=left] {$\displaystyle a_{4}$};
\draw (137.75,268.82) node [anchor=north west][inner sep=0.75pt]   [align=left] {$\displaystyle D_{a_{1}}$};
\draw (100.25,325.23) node [anchor=north west][inner sep=0.75pt]   [align=left] {$\displaystyle a_{2}$};
\draw (186.75,253.73) node [anchor=north west][inner sep=0.75pt]   [align=left] {$\displaystyle a_{3}$};
\draw (209.5,324.48) node [anchor=north west][inner sep=0.75pt]   [align=left] {$\displaystyle a_{4}$};

\end{tikzpicture}

	\end{center}
 \caption{Construction of the primitive curves for the Kishino's doodle}
 \label{kishino2}
\end{figure} 

From Equation \ref{eq:rela} and Figure \ref{kishino2}, we have that 
\[
\alpha(K)=\left( \begin{array}{r}
      -1 \\
      1\\
      1\\
      -1
      
\end{array}\right).
\]

Also from Figure \ref{kishino2}, we have that 
\[
\left(\left\langle \widetilde{\varphi}_{a_i}, \widetilde{\varphi}_{a_j} \right\rangle \right)_{i,j=1,2,3,4}=\left( \begin{array}{rrrr}
     0&1 &0&0 \\
     -1&0 &0&0\\
     0&0&0&-1\\
     0&0&1&0
\end{array} \right).
\]
Besides, 
\[
\left(\left\langle \varphi_{a_i}, \varphi_{D} \right\rangle-\left\langle \varphi_{a_j}, \varphi_{D} \right\rangle \right)_{i,j=1,2,3,4}=\left( \begin{array}{rrrr}
     0&-2 &-2&0 \\
     2&0 &0&2\\
     2&0&0&2\\
     0&-2&-2&0
\end{array} \right).
\]
From Equation \ref{eq:rela}, $\left(\left\langle \varphi_{a_i}, \varphi_{a_j} \right\rangle \right)_{i,j=1,2,3,4}=\left(\left\langle \widetilde{\varphi}_{a_i}, \widetilde{\varphi}_{a_j} \right\rangle \right)_{i,j=1,2,3,4}+\left(\left\langle \varphi_{a_i}, \varphi_{D} \right\rangle-\left\langle \varphi_{a_j}, \varphi_{D} \right\rangle \right)_{i,j=1,2,3,4}$. Thus,

\[
\left(\left\langle \varphi_{a_i}, \varphi_{a_j} \right\rangle \right)_{i,j=1,2,3,4}=\left( \begin{array}{rrrr}
     0&1 &0&0 \\
     -1&0 &0&0\\
     0&0&0&-1\\
     0&0&1&0
\end{array} \right)+\left( \begin{array}{rrrr}
     0&-2 &-2&0 \\
     2&0 &0&2\\
     2&0&0&2\\
     0&-2&-2&0
\end{array} \right)=\left(\begin{array}{rrrr}
0 & -1&-2&0\\
1 & 0&0&2\\
2&0&0&1\\
0&-2&-1&0
\end{array} \right).
\]

In this way,

 \[
   \lambda((T,D))=\left(\begin{array}{rrrr|r}
0 & -1&-2&0&-1\\
1 & 0&0&2&1\\
2&0&0&1&1\\
0&-2&-1&0&-1
\end{array} \right).
   \]
We observe that such skew-symmetric matrix is irreducible, so the Kishino's doodle is non classical.
\end{example}

\section{Virtualization of doodles}\label{sec:virtual}

Let $(\Sigma, D)$ be a surface doodle diagram, and let $a_i$ be a crossing point of $D$. Now, we consider a regular neighborhood $\mathcal{U}_{a_i} \subset \Sigma$ of $a_i$ such that no other crossing points of $D$ are in $\mathcal{U}_{a_i}$, as shown in Figure \ref{virtuproc}.
    \begin{definition}
        A virtualization of the doodle diagram $(\Sigma,D)$ at the crossing point $a_i$ is represented by the doodle diagram $(\widehat{\Sigma},v_{a_i}(D))$. Here, $\widehat{\Sigma}$ is a surface obtained by attaching a $1$-handle to $\Sigma$ within the neighborhood $\mathcal{U}$ disjoint from $D$. The doodle diagram $v_{a_i}(D)$ on $\widehat{\Sigma}$ is gotten from $D$ as illustrated in the following figure: 
    \end{definition}

    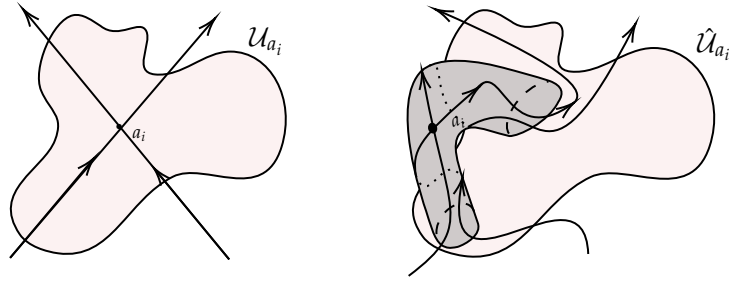
\begin{figure}[ht]
        \centering
      \tikzset{every picture/.style={line width=0.75pt}}

\begin{tikzpicture}[x=0.65pt,y=0.65pt,yscale=-1,xscale=1]

\draw [fill={rgb, 255:red, 252; green, 242; blue, 242 }  ,fill opacity=1 ][line width=0.75]    (96.45,59.83) .. controls (91.73,56.84) and (86.01,52.33) .. (80.36,53.56) .. controls (68.41,56.17) and (56.47,84.8) .. (61.79,97.42) .. controls (67.1,110.03) and (75.42,111.72) .. (69.83,126.4) .. controls (63.42,143.26) and (31.86,155.61) .. (49.41,181.22) .. controls (54.22,188.25) and (64.8,191.25) .. (71.69,192.19) .. controls (103.17,196.5) and (111.64,160.98) .. (138.54,147.55) .. controls (156.99,138.33) and (191.04,164.56) .. (201.68,130.32) .. controls (210.4,102.23) and (195.17,74.71) .. (171.96,74.71) .. controls (156.47,74.71) and (144.16,89.66) .. (128.02,88.02) .. controls (114.56,86.66) and (126.34,70.21) .. (123.07,62.96) .. controls (119.5,55.06) and (107.22,68.37) .. (96.45,59.83) -- cycle ;
\draw    (43.75,194.85) -- (91.5,137.99) ;
\draw [shift={(92.79,136.46)}, rotate = 130.03] [color={rgb, 255:red, 0; green, 0; blue, 0 }  ][line width=0.75]    (10.93,-3.29) .. controls (6.95,-1.4) and (3.31,-0.3) .. (0,0) .. controls (3.31,0.3) and (6.95,1.4) .. (10.93,3.29)   ;
\draw    (43.75,194.85) -- (161.95,56.58) ;
\draw [shift={(163.25,55.06)}, rotate = 130.53] [color={rgb, 255:red, 0; green, 0; blue, 0 }  ][line width=0.75]    (10.93,-3.29) .. controls (6.95,-1.4) and (3.31,-0.3) .. (0,0) .. controls (3.31,0.3) and (6.95,1.4) .. (10.93,3.29)   ;
\draw    (171.24,195.75) -- (128.44,143.8) ;
\draw [shift={(127.17,142.25)}, rotate = 50.52] [color={rgb, 255:red, 0; green, 0; blue, 0 }  ][line width=0.75]    (10.93,-3.29) .. controls (6.95,-1.4) and (3.31,-0.3) .. (0,0) .. controls (3.31,0.3) and (6.95,1.4) .. (10.93,3.29)   ;
\draw    (171.24,195.75) -- (52.21,51.19) ;
\draw [shift={(50.94,49.65)}, rotate = 50.53] [color={rgb, 255:red, 0; green, 0; blue, 0 }  ][line width=0.75]    (10.93,-3.29) .. controls (6.95,-1.4) and (3.31,-0.3) .. (0,0) .. controls (3.31,0.3) and (6.95,1.4) .. (10.93,3.29)   ;
\draw [fill={rgb, 255:red, 252; green, 242; blue, 242 }  ,fill opacity=1 ][line width=0.75]    (335.47,57.27) .. controls (330.29,54.19) and (324.03,49.53) .. (317.83,50.8) .. controls (304.74,53.49) and (291.65,83.03) .. (297.48,96.04) .. controls (303.31,109.05) and (312.42,110.79) .. (306.3,125.94) .. controls (299.27,143.32) and (264.69,156.06) .. (283.92,182.49) .. controls (289.19,189.74) and (300.78,192.83) .. (308.33,193.8) .. controls (342.83,198.24) and (352.11,161.61) .. (381.59,147.75) .. controls (401.81,138.24) and (439.12,165.3) .. (450.77,129.98) .. controls (460.33,101) and (443.65,72.62) .. (418.22,72.62) .. controls (401.24,72.62) and (387.75,88.04) .. (370.06,86.35) .. controls (355.31,84.94) and (368.22,67.97) .. (364.63,60.5) .. controls (360.72,52.35) and (347.27,66.07) .. (335.47,57.27) -- cycle ;
\draw [fill={rgb, 255:red, 207; green, 202; blue, 202 }  ,fill opacity=1 ][line width=0.75] [line join = round][line cap = round]   (295.09,187.66) .. controls (290.1,183.88) and (288.16,174.44) .. (285.84,167.6) .. controls (283.38,160.37) and (264.42,104.1) .. (283.37,85.74) .. controls (297.21,73.96) and (330.17,84.02) .. (333.12,84.78) .. controls (337.49,85.91) and (349.36,89.09) .. (357.76,91.84) .. controls (365.21,94.64) and (366.21,98.99) .. (362.73,104.08) .. controls (356.99,115.67) and (339.91,124.11) .. (334.27,123.91) .. controls (318.06,120.01) and (311.53,115.59) .. (305.41,122.48) .. controls (300.11,128.75) and (305.34,141.89) .. (305.76,144.49) .. controls (311.7,161.94) and (310.01,157.14) .. (315.94,174.59) .. controls (318.23,181.34) and (309.12,186.83) .. (304.15,188.46) .. controls (299.03,189.45) and (298.08,189.1) .. (295.09,187.66) -- cycle ;
\draw  [draw opacity=0][dash pattern={on 4.5pt off 4.5pt}][line width=0.75]  (291.89,181) .. controls (291.46,176.83) and (293.48,171.64) .. (297.38,167.76) .. controls (302.72,162.45) and (309.58,161.59) .. (312.7,165.84) .. controls (313.67,167.16) and (314.17,168.82) .. (314.23,170.64) -- (303.03,175.45) -- cycle ; \draw  [dash pattern={on 4.5pt off 4.5pt}][line width=0.75]  (291.89,181) .. controls (291.46,176.83) and (293.48,171.64) .. (297.38,167.76) .. controls (302.72,162.45) and (309.58,161.59) .. (312.7,165.84) .. controls (313.67,167.16) and (314.17,168.82) .. (314.23,170.64) ;  
\draw  [draw opacity=0][dash pattern={on 4.5pt off 4.5pt}] (334.05,122.48) .. controls (330.78,117.81) and (334.35,107.88) .. (342.08,100.19) .. controls (347.17,95.13) and (352.81,92.29) .. (357.05,92.22) -- (348.23,108.55) -- cycle ; \draw  [dash pattern={on 4.5pt off 4.5pt}] (334.05,122.48) .. controls (330.78,117.81) and (334.35,107.88) .. (342.08,100.19) .. controls (347.17,95.13) and (352.81,92.29) .. (357.05,92.22) ;  
\draw  [fill={rgb, 255:red, 0; green, 0; blue, 0 }  ,fill opacity=1 ] (287.65,119.9) .. controls (287.24,118.53) and (287.78,117.15) .. (288.87,116.82) .. controls (289.96,116.49) and (291.17,117.33) .. (291.59,118.7) .. controls (292,120.07) and (291.46,121.45) .. (290.37,121.78) .. controls (289.29,122.11) and (288.07,121.27) .. (287.65,119.9) -- cycle ;
\draw  [dash pattern={on 0.84pt off 2.51pt}]  (279.83,146.59) .. controls (287.17,168.25) and (296.83,128.59) .. (307.5,149.92) ;
\draw    (279.83,146.59) .. controls (277.85,133.06) and (280.45,125.73) .. (316.08,95.83) ;
\draw [shift={(317.17,94.92)}, rotate = 140.09] [color={rgb, 255:red, 0; green, 0; blue, 0 }  ][line width=0.75]    (10.93,-3.29) .. controls (6.95,-1.4) and (3.31,-0.3) .. (0,0) .. controls (3.31,0.3) and (6.95,1.4) .. (10.93,3.29)   ;
\draw    (380.17,192.59) .. controls (377.85,134.55) and (291.04,231.28) .. (306.92,148.84) ;
\draw [shift={(307.17,147.59)}, rotate = 101.35] [color={rgb, 255:red, 0; green, 0; blue, 0 }  ][line width=0.75]    (10.93,-3.29) .. controls (6.95,-1.4) and (3.31,-0.3) .. (0,0) .. controls (3.31,0.3) and (6.95,1.4) .. (10.93,3.29)   ;
\draw    (275.67,205.34) .. controls (315.47,175.49) and (297.68,168.33) .. (283.58,86.97) ;
\draw [shift={(283.37,85.74)}, rotate = 80.28] [color={rgb, 255:red, 0; green, 0; blue, 0 }  ][line width=0.75]    (10.93,-3.29) .. controls (6.95,-1.4) and (3.31,-0.3) .. (0,0) .. controls (3.31,0.3) and (6.95,1.4) .. (10.93,3.29)   ;
\draw  [dash pattern={on 0.84pt off 2.51pt}]  (283.37,85.74) .. controls (301.17,71.59) and (283.83,105.59) .. (312.17,118.92) ;
\draw    (310.17,118.59) .. controls (354.28,93.05) and (351.2,169.48) .. (406.99,59.92) ;
\draw [shift={(407.83,58.25)}, rotate = 116.84] [color={rgb, 255:red, 0; green, 0; blue, 0 }  ][line width=0.75]    (10.93,-3.29) .. controls (6.95,-1.4) and (3.31,-0.3) .. (0,0) .. controls (3.31,0.3) and (6.95,1.4) .. (10.93,3.29)   ;
\draw    (314.5,97.25) .. controls (330.34,76.13) and (332.13,133.1) .. (370.98,105.13) ;
\draw [shift={(372.17,104.25)}, rotate = 143.13] [color={rgb, 255:red, 0; green, 0; blue, 0 }  ][line width=0.75]    (10.93,-3.29) .. controls (6.95,-1.4) and (3.31,-0.3) .. (0,0) .. controls (3.31,0.3) and (6.95,1.4) .. (10.93,3.29)   ;
\draw    (369.5,106.25) .. controls (384.43,97.96) and (360.74,80.43) .. (291.87,50.37) ;
\draw [shift={(290.83,49.92)}, rotate = 23.53] [color={rgb, 255:red, 0; green, 0; blue, 0 }  ][line width=0.75]    (10.93,-3.29) .. controls (6.95,-1.4) and (3.31,-0.3) .. (0,0) .. controls (3.31,0.3) and (6.95,1.4) .. (10.93,3.29)   ;

\draw (297.57,112.74) node [anchor=north west][inner sep=0.75pt]  [font=\scriptsize,rotate=-343.14] [align=left] {$\displaystyle a_{i}$};
\draw (181.04,58.96) node [anchor=north west][inner sep=0.75pt]   [align=left] {$\displaystyle \mathcal{U}_{a_{i}} $};
\draw (103.36,115.19) node [anchor=north west][inner sep=0.75pt]  [font=\tiny] [align=left] {$\displaystyle \bullet $};
\draw (113.13,118.03) node [anchor=north west][inner sep=0.75pt]  [font=\scriptsize] [align=left] {$\displaystyle a_{i}$};
\draw (441.79,59.24) node [anchor=north west][inner sep=0.75pt]   [align=left] {$\displaystyle \hat{\mathcal{U}}_{a_{i}}{} \ $};
\end{tikzpicture}
        \caption{Virtualization process at the crossing point $a_i$}
        \label{virtuproc}
    \end{figure}

We do not lose generality by assuming that the virtualization occurs at the crossing point $a_n$. To simplify the notation, we use $V(D)$ instead of $v_{a_n}(D)$. It is not hard to prove that, for every $a_j\neq a_i$, $V(D_{a_j})=V(D)_{a_j}$.

\begin{proposition}\label{comple}
    Let $(\Sigma,D)$ be a one-component doodle diagram and let  $(\widehat{\Sigma},D)$ be the doodle diagram obtained by attaching a $1$-handle inside of the neighborhood $\mathcal{U}_{a_n}$ disjoint of the curve $D$ as described in Figure \ref{imm}. Then, $(\Sigma,D)$ is stably $R$-equivalent to $(\widehat{\Sigma},D)$. Therefore, $(\Sigma,D_{a_k})$ is stably $R$-equivalent to $(\widehat{\Sigma},D_{a_k})$, for all $k=1,\ldots,n$. Besides,
    \begin{enumerate}
    \item $D\cap \left( \widehat{\Sigma}\setminus \widehat{U}_{a_n}\right)=V(D)\cap \left( \widehat{\Sigma}\setminus \widehat{U}_{a_n}\right)$.
        \item For all $a_j \neq a_n$, $D_{a_j}\cap \left( \widehat{\Sigma}\setminus \widehat{U}_{a_n}\right)=V(D)_{a_j}\cap \left( \widehat{\Sigma}\setminus \widehat{U}_{a_n}\right).$ Moreover, if $D_{a_j}$ does not travel the crossing point $a_i$, then $D_{a_j}=V(D)_{a_j}$.
        \item $\widetilde{D}_{a_n}\cap \left( \widehat{\Sigma}\setminus \widehat{U}_{a_n}\right)=V(D)_{a_n}\cap \left( \widehat{\Sigma}\setminus \widehat{U}_{a_n}\right).$
        \item $(\widehat{\Sigma},V(D)_{a_n})$ is stably R-equivalent to $(\widehat{\Sigma},\widetilde{D}_{a_n})$ and  $(\widehat{\Sigma},\widetilde{V(D)}_{a_n})$ is stably R-equivalent to $(\widehat{\Sigma},D_{a_n})$.
        
    \end{enumerate}

    \begin{figure}[ht]
        \centering
        \tikzset{every picture/.style={line width=0.75pt}} 

\begin{tikzpicture}[x=0.65pt,y=0.65pt,yscale=-1,xscale=1]

\draw [fill={rgb, 255:red, 252; green, 242; blue, 242 }  ,fill opacity=1 ][line width=0.75]    (108.45,308.83) .. controls (103.73,305.84) and (98.01,301.33) .. (92.36,302.56) .. controls (80.41,305.17) and (68.47,333.8) .. (73.79,346.42) .. controls (79.1,359.03) and (87.42,360.72) .. (81.83,375.4) .. controls (75.42,392.26) and (43.86,404.61) .. (61.41,430.22) .. controls (66.22,437.25) and (76.8,440.25) .. (83.69,441.19) .. controls (115.17,445.5) and (123.64,409.98) .. (150.54,396.55) .. controls (168.99,387.33) and (203.04,413.56) .. (213.68,379.32) .. controls (222.4,351.23) and (207.17,323.71) .. (183.96,323.71) .. controls (168.47,323.71) and (156.16,338.66) .. (140.02,337.02) .. controls (126.56,335.66) and (138.34,319.21) .. (135.07,311.96) .. controls (131.5,304.06) and (119.22,317.37) .. (108.45,308.83) -- cycle ;
\draw    (55.75,443.85) -- (103.5,386.99) ;
\draw [shift={(104.79,385.46)}, rotate = 130.03] [color={rgb, 255:red, 0; green, 0; blue, 0 }  ][line width=0.75]    (10.93,-3.29) .. controls (6.95,-1.4) and (3.31,-0.3) .. (0,0) .. controls (3.31,0.3) and (6.95,1.4) .. (10.93,3.29)   ;
\draw    (55.75,443.85) -- (173.95,305.58) ;
\draw [shift={(175.25,304.06)}, rotate = 130.53] [color={rgb, 255:red, 0; green, 0; blue, 0 }  ][line width=0.75]    (10.93,-3.29) .. controls (6.95,-1.4) and (3.31,-0.3) .. (0,0) .. controls (3.31,0.3) and (6.95,1.4) .. (10.93,3.29)   ;
\draw    (183.24,444.75) -- (140.44,392.8) ;
\draw [shift={(139.17,391.25)}, rotate = 50.52] [color={rgb, 255:red, 0; green, 0; blue, 0 }  ][line width=0.75]    (10.93,-3.29) .. controls (6.95,-1.4) and (3.31,-0.3) .. (0,0) .. controls (3.31,0.3) and (6.95,1.4) .. (10.93,3.29)   ;
\draw    (183.24,444.75) -- (64.21,300.19) ;
\draw [shift={(62.94,298.65)}, rotate = 50.53] [color={rgb, 255:red, 0; green, 0; blue, 0 }  ][line width=0.75]    (10.93,-3.29) .. controls (6.95,-1.4) and (3.31,-0.3) .. (0,0) .. controls (3.31,0.3) and (6.95,1.4) .. (10.93,3.29)   ;
\draw [fill={rgb, 255:red, 252; green, 242; blue, 242 }  ,fill opacity=1 ][line width=0.75]    (347.47,306.27) .. controls (342.29,303.19) and (336.03,298.53) .. (329.83,299.8) .. controls (316.74,302.49) and (303.65,332.03) .. (309.48,345.04) .. controls (315.31,358.05) and (324.42,359.79) .. (318.3,374.94) .. controls (311.27,392.32) and (276.69,405.06) .. (295.92,431.49) .. controls (301.19,438.74) and (312.78,441.83) .. (320.33,442.8) .. controls (354.83,447.24) and (364.11,410.61) .. (393.59,396.75) .. controls (413.81,387.24) and (451.12,414.3) .. (462.77,378.98) .. controls (472.33,350) and (455.65,321.62) .. (430.22,321.62) .. controls (413.24,321.62) and (399.75,337.04) .. (382.06,335.35) .. controls (367.31,333.94) and (380.22,316.97) .. (376.63,309.5) .. controls (372.72,301.35) and (359.27,315.07) .. (347.47,306.27) -- cycle ;
\draw [fill={rgb, 255:red, 207; green, 202; blue, 202 }  ,fill opacity=1 ][line width=0.75] [line join = round][line cap = round]   (307.09,436.66) .. controls (302.1,432.88) and (300.16,423.44) .. (297.84,416.6) .. controls (295.38,409.37) and (276.42,353.1) .. (295.37,334.74) .. controls (309.21,322.96) and (342.17,333.02) .. (345.12,333.78) .. controls (349.49,334.91) and (361.36,338.09) .. (369.76,340.84) .. controls (377.21,343.64) and (378.21,347.99) .. (374.73,353.08) .. controls (368.99,364.67) and (351.91,373.11) .. (346.27,372.91) .. controls (330.06,369.01) and (323.53,364.59) .. (317.41,371.48) .. controls (312.11,377.75) and (317.34,390.89) .. (317.76,393.49) .. controls (323.7,410.94) and (322.01,406.14) .. (327.94,423.59) .. controls (330.23,430.34) and (321.12,435.83) .. (316.15,437.46) .. controls (311.03,438.45) and (310.08,438.1) .. (307.09,436.66) -- cycle ;
\draw  [draw opacity=0][dash pattern={on 4.5pt off 4.5pt}][line width=0.75]  (303.89,430) .. controls (303.46,425.83) and (305.48,420.64) .. (309.38,416.76) .. controls (314.72,411.45) and (321.58,410.59) .. (324.7,414.84) .. controls (325.67,416.16) and (326.17,417.82) .. (326.23,419.64) -- (315.03,424.45) -- cycle ; \draw  [dash pattern={on 4.5pt off 4.5pt}][line width=0.75]  (303.89,430) .. controls (303.46,425.83) and (305.48,420.64) .. (309.38,416.76) .. controls (314.72,411.45) and (321.58,410.59) .. (324.7,414.84) .. controls (325.67,416.16) and (326.17,417.82) .. (326.23,419.64) ;  
\draw  [draw opacity=0][dash pattern={on 4.5pt off 4.5pt}] (346.05,371.48) .. controls (342.78,366.81) and (346.35,356.88) .. (354.08,349.19) .. controls (359.17,344.13) and (364.81,341.29) .. (369.05,341.22) -- (360.23,357.55) -- cycle ; \draw  [dash pattern={on 4.5pt off 4.5pt}] (346.05,371.48) .. controls (342.78,366.81) and (346.35,356.88) .. (354.08,349.19) .. controls (359.17,344.13) and (364.81,341.29) .. (369.05,341.22) ;  
\draw  [fill={rgb, 255:red, 0; green, 0; blue, 0 }  ,fill opacity=1 ] (299.65,368.9) .. controls (299.24,367.53) and (299.78,366.15) .. (300.87,365.82) .. controls (301.96,365.49) and (303.17,366.33) .. (303.59,367.7) .. controls (304,369.07) and (303.46,370.45) .. (302.37,370.78) .. controls (301.29,371.11) and (300.07,370.27) .. (299.65,368.9) -- cycle ;
\draw    (318.13,451.89) .. controls (315.81,393.85) and (255,375.57) .. (337.14,355.69) ;
\draw [shift={(338.38,355.39)}, rotate = 166.61] [color={rgb, 255:red, 0; green, 0; blue, 0 }  ][line width=0.75]    (10.93,-3.29) .. controls (6.95,-1.4) and (3.31,-0.3) .. (0,0) .. controls (3.31,0.3) and (6.95,1.4) .. (10.93,3.29)   ;
\draw    (388.63,436.39) .. controls (354.88,391.89) and (350.8,440.2) .. (331.88,433.89) .. controls (312.97,427.58) and (306.89,392.8) .. (301.62,368.3) .. controls (296.43,344.17) and (308.27,333.94) .. (345.27,339.85) ;
\draw [shift={(346.98,340.13)}, rotate = 189.62] [color={rgb, 255:red, 0; green, 0; blue, 0 }  ][line width=0.75]    (10.93,-3.29) .. controls (6.95,-1.4) and (3.31,-0.3) .. (0,0) .. controls (3.31,0.3) and (6.95,1.4) .. (10.93,3.29)   ;
\draw    (336.92,356.09) .. controls (381.03,330.55) and (377.95,406.98) .. (433.74,297.42) ;
\draw [shift={(434.58,295.75)}, rotate = 116.84] [color={rgb, 255:red, 0; green, 0; blue, 0 }  ][line width=0.75]    (10.93,-3.29) .. controls (6.95,-1.4) and (3.31,-0.3) .. (0,0) .. controls (3.31,0.3) and (6.95,1.4) .. (10.93,3.29)   ;
\draw    (343.88,339.76) .. controls (411.54,359.16) and (372.89,329.55) .. (303.88,299.38) ;
\draw [shift={(302.83,298.92)}, rotate = 23.53] [color={rgb, 255:red, 0; green, 0; blue, 0 }  ][line width=0.75]    (10.93,-3.29) .. controls (6.95,-1.4) and (3.31,-0.3) .. (0,0) .. controls (3.31,0.3) and (6.95,1.4) .. (10.93,3.29)   ;
\draw    (306.88,393.39) -- (303.81,379.09) ;
\draw [shift={(303.38,377.14)}, rotate = 77.85] [color={rgb, 255:red, 0; green, 0; blue, 0 }  ][line width=0.75]    (10.93,-3.29) .. controls (6.95,-1.4) and (3.31,-0.3) .. (0,0) .. controls (3.31,0.3) and (6.95,1.4) .. (10.93,3.29)   ;
\draw    (302.63,403.39) -- (298.01,393.93) ;
\draw [shift={(297.13,392.14)}, rotate = 63.95] [color={rgb, 255:red, 0; green, 0; blue, 0 }  ][line width=0.75]    (10.93,-3.29) .. controls (6.95,-1.4) and (3.31,-0.3) .. (0,0) .. controls (3.31,0.3) and (6.95,1.4) .. (10.93,3.29)   ;

\draw (453.79,308.24) node [anchor=north west][inner sep=0.75pt]   [align=left] {$\displaystyle \hat{\mathcal{U}}_{a_{n}}{} \ $};
\draw (128.13,362.78) node [anchor=north west][inner sep=0.75pt]  [font=\scriptsize] [align=left] {$\displaystyle a_{n}$};
\draw (115.36,365.19) node [anchor=north west][inner sep=0.75pt]  [font=\tiny] [align=left] {$\displaystyle \bullet $};
\draw (193.04,307.96) node [anchor=north west][inner sep=0.75pt]   [align=left] {$\displaystyle \mathcal{U}_{a_{n}}{} \ $};
\draw (301.57,369.74) node [anchor=north west][inner sep=0.75pt]  [font=\scriptsize,rotate=-343.14] [align=left] {$\displaystyle a_{n}$};
\draw (399,409.63) node [anchor=north west][inner sep=0.75pt]   [align=left] {$\displaystyle (\hat{\Sigma } ,D)$};
\draw (120.25,425) node [anchor=north west][inner sep=0.75pt]   [align=left] {$\displaystyle ( \Sigma ,D)$};

\end{tikzpicture}

        \caption{Immersion of $D$ into the surface $\widehat{\Sigma}$.}
        \label{imm}
    \end{figure}
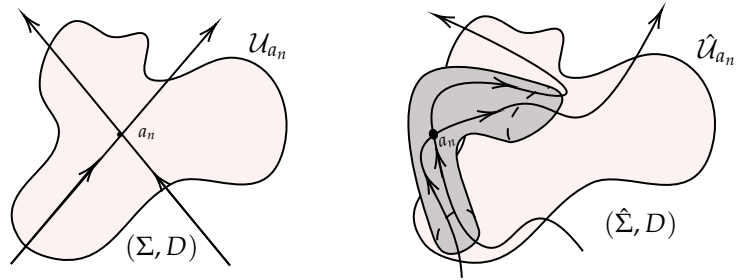
\end{proposition}
\begin{proof}
    The proof of $(a)$ and $(b)$ comes from the fact that the virtualization of a doodle diagram only changes the part of the diagram inside the neighborhood $\mathcal{U}_{a_n}$, while the rest of the diagram remains unchanged. The proof of the literals $(c)$ and $(d)$ are gotten from  Figure \ref{primivirt}. 

    \begin{figure}[ht]
        \centering
      \tikzset{every picture/.style={line width=0.75pt}} 

\begin{tikzpicture}[x=0.65pt,y=0.65pt,yscale=-1,xscale=1]

\draw [fill={rgb, 255:red, 252; green, 242; blue, 242 }  ,fill opacity=1 ][line width=0.75]    (101.22,66.38) .. controls (96.04,63.3) and (89.78,58.64) .. (83.58,59.92) .. controls (70.49,62.61) and (57.4,92.14) .. (63.23,105.16) .. controls (69.06,118.17) and (78.17,119.91) .. (72.05,135.05) .. controls (65.02,152.44) and (30.44,165.18) .. (49.67,191.6) .. controls (54.94,198.85) and (66.53,201.94) .. (74.08,202.91) .. controls (108.58,207.36) and (117.86,170.73) .. (147.34,156.86) .. controls (167.56,147.36) and (204.87,174.42) .. (216.52,139.09) .. controls (226.08,110.12) and (209.4,81.73) .. (183.97,81.73) .. controls (166.99,81.73) and (153.5,97.15) .. (135.81,95.47) .. controls (121.06,94.06) and (133.97,77.09) .. (130.38,69.61) .. controls (126.47,61.47) and (113.02,75.19) .. (101.22,66.38) -- cycle ;
\draw [fill={rgb, 255:red, 207; green, 202; blue, 202 }  ,fill opacity=1 ][line width=0.75] [line join = round][line cap = round]   (60.13,197.16) .. controls (55.14,193.39) and (53.2,183.94) .. (50.88,177.1) .. controls (48.42,169.87) and (29.46,113.6) .. (48.4,95.24) .. controls (62.25,83.46) and (95.21,93.53) .. (98.15,94.28) .. controls (102.53,95.41) and (114.39,98.59) .. (122.8,101.34) .. controls (130.25,104.14) and (131.25,108.49) .. (127.76,113.58) .. controls (122.03,125.18) and (104.95,133.61) .. (99.31,133.41) .. controls (83.1,129.51) and (76.56,125.09) .. (70.45,131.98) .. controls (65.15,138.25) and (70.38,151.39) .. (70.8,153.99) .. controls (76.74,171.44) and (75.04,166.64) .. (80.98,184.09) .. controls (83.27,190.84) and (74.16,196.33) .. (69.19,197.96) .. controls (64.07,198.95) and (63.12,198.6) .. (60.13,197.16) -- cycle ;
\draw  [draw opacity=0][dash pattern={on 4.5pt off 4.5pt}][line width=0.75]  (57.64,190.12) .. controls (57.21,185.95) and (59.23,180.76) .. (63.13,176.87) .. controls (68.47,171.56) and (75.33,170.7) .. (78.45,174.96) .. controls (79.42,176.28) and (79.92,177.94) .. (79.98,179.76) -- (68.78,184.57) -- cycle ; \draw  [dash pattern={on 4.5pt off 4.5pt}][line width=0.75]  (57.64,190.12) .. controls (57.21,185.95) and (59.23,180.76) .. (63.13,176.87) .. controls (68.47,171.56) and (75.33,170.7) .. (78.45,174.96) .. controls (79.42,176.28) and (79.92,177.94) .. (79.98,179.76) ;  
\draw  [draw opacity=0][dash pattern={on 4.5pt off 4.5pt}] (99.8,131.59) .. controls (96.53,126.93) and (100.1,116.99) .. (107.83,109.3) .. controls (112.92,104.24) and (118.56,101.41) .. (122.8,101.34) -- (113.98,117.67) -- cycle ; \draw  [dash pattern={on 4.5pt off 4.5pt}] (99.8,131.59) .. controls (96.53,126.93) and (100.1,116.99) .. (107.83,109.3) .. controls (112.92,104.24) and (118.56,101.41) .. (122.8,101.34) ;  
\draw  [fill={rgb, 255:red, 0; green, 0; blue, 0 }  ,fill opacity=1 ] (53.4,129.01) .. controls (52.99,127.64) and (53.53,126.26) .. (54.62,125.93) .. controls (55.71,125.6) and (56.92,126.45) .. (57.34,127.82) .. controls (57.75,129.19) and (57.21,130.57) .. (56.12,130.9) .. controls (55.04,131.23) and (53.82,130.38) .. (53.4,129.01) -- cycle ;
\draw    (71.88,212) .. controls (69.56,153.96) and (8.75,135.68) .. (90.89,115.8) ;
\draw [shift={(92.13,115.5)}, rotate = 166.61] [color={rgb, 255:red, 0; green, 0; blue, 0 }  ][line width=0.75]    (10.93,-3.29) .. controls (6.95,-1.4) and (3.31,-0.3) .. (0,0) .. controls (3.31,0.3) and (6.95,1.4) .. (10.93,3.29)   ;
\draw    (142.38,196.5) .. controls (108.63,152) and (104.55,200.31) .. (85.63,194) .. controls (66.72,187.69) and (60.64,152.91) .. (55.37,128.42) .. controls (50.18,104.29) and (62.02,94.06) .. (99.02,99.97) ;
\draw [shift={(100.73,100.25)}, rotate = 189.62] [color={rgb, 255:red, 0; green, 0; blue, 0 }  ][line width=0.75]    (10.93,-3.29) .. controls (6.95,-1.4) and (3.31,-0.3) .. (0,0) .. controls (3.31,0.3) and (6.95,1.4) .. (10.93,3.29)   ;
\draw    (90.67,116.2) .. controls (134.78,90.66) and (131.7,167.1) .. (187.49,57.53) ;
\draw [shift={(188.33,55.87)}, rotate = 116.84] [color={rgb, 255:red, 0; green, 0; blue, 0 }  ][line width=0.75]    (10.93,-3.29) .. controls (6.95,-1.4) and (3.31,-0.3) .. (0,0) .. controls (3.31,0.3) and (6.95,1.4) .. (10.93,3.29)   ;
\draw    (97.63,99.88) .. controls (165.29,119.28) and (126.64,89.67) .. (57.63,59.49) ;
\draw [shift={(56.58,59.04)}, rotate = 23.53] [color={rgb, 255:red, 0; green, 0; blue, 0 }  ][line width=0.75]    (10.93,-3.29) .. controls (6.95,-1.4) and (3.31,-0.3) .. (0,0) .. controls (3.31,0.3) and (6.95,1.4) .. (10.93,3.29)   ;
\draw    (60.63,153.5) -- (57.56,139.21) ;
\draw [shift={(57.13,137.25)}, rotate = 77.85] [color={rgb, 255:red, 0; green, 0; blue, 0 }  ][line width=0.75]    (10.93,-3.29) .. controls (6.95,-1.4) and (3.31,-0.3) .. (0,0) .. controls (3.31,0.3) and (6.95,1.4) .. (10.93,3.29)   ;
\draw    (56.38,163.5) -- (51.76,154.05) ;
\draw [shift={(50.88,152.25)}, rotate = 63.95] [color={rgb, 255:red, 0; green, 0; blue, 0 }  ][line width=0.75]    (10.93,-3.29) .. controls (6.95,-1.4) and (3.31,-0.3) .. (0,0) .. controls (3.31,0.3) and (6.95,1.4) .. (10.93,3.29)   ;
\draw [color={rgb, 255:red, 208; green, 2; blue, 27 }  ,draw opacity=1 ]   (67.15,231.75) .. controls (73.95,203.75) and (51.67,171.54) .. (46.33,146.27) .. controls (41,120.99) and (45.33,118.93) .. (47,110.93) .. controls (48.67,102.93) and (56.4,93.42) .. (69,92.93) .. controls (81.6,92.44) and (134.6,107.13) .. (130.33,102.6) .. controls (126.39,98.41) and (71.21,71.4) .. (60.64,66.1) ;
\draw [shift={(59,65.27)}, rotate = 28.71] [color={rgb, 255:red, 208; green, 2; blue, 27 }  ,draw opacity=1 ][line width=0.75]    (10.93,-3.29) .. controls (6.95,-1.4) and (3.31,-0.3) .. (0,0) .. controls (3.31,0.3) and (6.95,1.4) .. (10.93,3.29)   ;
\draw [color={rgb, 255:red, 208; green, 2; blue, 27 }  ,draw opacity=1 ]   (134.33,198.33) .. controls (135.45,190.51) and (128.97,183.67) .. (118.33,181.33) .. controls (107.69,179) and (99.27,198.21) .. (88.67,198.67) .. controls (78.06,199.12) and (75.14,193.13) .. (68.67,181.67) .. controls (62.19,170.21) and (39.18,130.99) .. (56.33,116.67) .. controls (73.48,102.34) and (89.31,108.03) .. (110.67,108.67) .. controls (132.02,109.3) and (136.69,121.21) .. (143.67,114.33) .. controls (150.33,107.76) and (175.28,67.84) .. (180.7,54.39) ;
\draw [shift={(181.33,52.67)}, rotate = 107.93] [color={rgb, 255:red, 208; green, 2; blue, 27 }  ,draw opacity=1 ][line width=0.75]    (10.93,-3.29) .. controls (6.95,-1.4) and (3.31,-0.3) .. (0,0) .. controls (3.31,0.3) and (6.95,1.4) .. (10.93,3.29)   ;
\draw [fill={rgb, 255:red, 252; green, 242; blue, 242 }  ,fill opacity=1 ][line width=0.75]    (362.13,55.27) .. controls (356.96,52.19) and (350.69,47.53) .. (344.5,48.8) .. controls (331.41,51.49) and (318.32,81.03) .. (324.15,94.04) .. controls (329.97,107.05) and (339.09,108.79) .. (332.97,123.94) .. controls (325.94,141.32) and (291.35,154.06) .. (310.58,180.49) .. controls (315.86,187.74) and (327.45,190.83) .. (335,191.8) .. controls (369.5,196.24) and (378.77,159.61) .. (408.25,145.75) .. controls (428.47,136.24) and (465.79,163.3) .. (477.44,127.98) .. controls (487,99) and (470.31,70.62) .. (444.88,70.62) .. controls (427.91,70.62) and (414.42,86.04) .. (396.72,84.35) .. controls (381.98,82.94) and (394.88,65.97) .. (391.3,58.5) .. controls (387.39,50.35) and (373.93,64.07) .. (362.13,55.27) -- cycle ;
\draw [fill={rgb, 255:red, 207; green, 202; blue, 202 }  ,fill opacity=1 ][line width=0.75] [line join = round][line cap = round]   (321.76,185.66) .. controls (316.77,181.88) and (314.83,172.44) .. (312.51,165.6) .. controls (310.05,158.37) and (291.08,102.1) .. (310.03,83.74) .. controls (323.88,71.96) and (356.83,82.02) .. (359.78,82.78) .. controls (364.16,83.91) and (376.02,87.09) .. (384.43,89.84) .. controls (391.87,92.64) and (392.88,96.99) .. (389.39,102.08) .. controls (383.66,113.67) and (366.58,122.11) .. (360.94,121.91) .. controls (344.72,118.01) and (338.19,113.59) .. (332.08,120.48) .. controls (326.78,126.75) and (332.01,139.89) .. (332.43,142.49) .. controls (338.37,159.94) and (336.67,155.14) .. (342.6,172.59) .. controls (344.9,179.34) and (335.79,184.83) .. (330.81,186.46) .. controls (325.7,187.45) and (324.75,187.1) .. (321.76,185.66) -- cycle ;
\draw  [draw opacity=0][dash pattern={on 4.5pt off 4.5pt}][line width=0.75]  (318.55,179) .. controls (318.13,174.83) and (320.14,169.64) .. (324.05,165.76) .. controls (329.39,160.45) and (336.25,159.59) .. (339.37,163.84) .. controls (340.34,165.16) and (340.83,166.82) .. (340.89,168.64) -- (329.7,173.45) -- cycle ; \draw  [dash pattern={on 4.5pt off 4.5pt}][line width=0.75]  (318.55,179) .. controls (318.13,174.83) and (320.14,169.64) .. (324.05,165.76) .. controls (329.39,160.45) and (336.25,159.59) .. (339.37,163.84) .. controls (340.34,165.16) and (340.83,166.82) .. (340.89,168.64) ;  
\draw  [draw opacity=0][dash pattern={on 4.5pt off 4.5pt}] (360.72,120.48) .. controls (357.44,115.81) and (361.01,105.88) .. (368.75,98.19) .. controls (373.84,93.13) and (379.48,90.29) .. (383.72,90.22) -- (374.89,106.55) -- cycle ; \draw  [dash pattern={on 4.5pt off 4.5pt}] (360.72,120.48) .. controls (357.44,115.81) and (361.01,105.88) .. (368.75,98.19) .. controls (373.84,93.13) and (379.48,90.29) .. (383.72,90.22) ;  
\draw  [fill={rgb, 255:red, 0; green, 0; blue, 0 }  ,fill opacity=1 ] (314.32,117.9) .. controls (313.9,116.53) and (314.45,115.15) .. (315.54,114.82) .. controls (316.62,114.49) and (317.84,115.33) .. (318.26,116.7) .. controls (318.67,118.07) and (318.13,119.45) .. (317.04,119.78) .. controls (315.95,120.11) and (314.73,119.27) .. (314.32,117.9) -- cycle ;
\draw  [dash pattern={on 0.84pt off 2.51pt}]  (306.5,144.59) .. controls (313.83,166.25) and (323.5,126.59) .. (334.17,147.92) ;
\draw    (306.5,144.59) .. controls (304.52,131.06) and (307.11,123.73) .. (342.74,93.83) ;
\draw [shift={(343.83,92.92)}, rotate = 140.09] [color={rgb, 255:red, 0; green, 0; blue, 0 }  ][line width=0.75]    (10.93,-3.29) .. controls (6.95,-1.4) and (3.31,-0.3) .. (0,0) .. controls (3.31,0.3) and (6.95,1.4) .. (10.93,3.29)   ;
\draw    (406.83,190.59) .. controls (404.51,132.55) and (317.71,229.28) .. (333.59,146.84) ;
\draw [shift={(333.83,145.59)}, rotate = 101.35] [color={rgb, 255:red, 0; green, 0; blue, 0 }  ][line width=0.75]    (10.93,-3.29) .. controls (6.95,-1.4) and (3.31,-0.3) .. (0,0) .. controls (3.31,0.3) and (6.95,1.4) .. (10.93,3.29)   ;
\draw    (302.33,203.34) .. controls (342.13,173.49) and (324.35,166.33) .. (310.24,84.97) ;
\draw [shift={(310.03,83.74)}, rotate = 80.28] [color={rgb, 255:red, 0; green, 0; blue, 0 }  ][line width=0.75]    (10.93,-3.29) .. controls (6.95,-1.4) and (3.31,-0.3) .. (0,0) .. controls (3.31,0.3) and (6.95,1.4) .. (10.93,3.29)   ;
\draw  [dash pattern={on 0.84pt off 2.51pt}]  (310.03,83.74) .. controls (327.83,69.59) and (310.5,103.59) .. (338.83,116.92) ;
\draw    (336.83,116.59) .. controls (380.95,91.05) and (377.87,167.48) .. (433.65,57.92) ;
\draw [shift={(434.5,56.25)}, rotate = 116.84] [color={rgb, 255:red, 0; green, 0; blue, 0 }  ][line width=0.75]    (10.93,-3.29) .. controls (6.95,-1.4) and (3.31,-0.3) .. (0,0) .. controls (3.31,0.3) and (6.95,1.4) .. (10.93,3.29)   ;
\draw    (341.17,95.25) .. controls (357.01,74.13) and (358.8,131.1) .. (397.64,103.13) ;
\draw [shift={(398.83,102.25)}, rotate = 143.13] [color={rgb, 255:red, 0; green, 0; blue, 0 }  ][line width=0.75]    (10.93,-3.29) .. controls (6.95,-1.4) and (3.31,-0.3) .. (0,0) .. controls (3.31,0.3) and (6.95,1.4) .. (10.93,3.29)   ;
\draw    (396.17,104.25) .. controls (411.09,95.96) and (387.41,78.43) .. (318.54,48.37) ;
\draw [shift={(317.5,47.92)}, rotate = 23.53] [color={rgb, 255:red, 0; green, 0; blue, 0 }  ][line width=0.75]    (10.93,-3.29) .. controls (6.95,-1.4) and (3.31,-0.3) .. (0,0) .. controls (3.31,0.3) and (6.95,1.4) .. (10.93,3.29)   ;

\draw [color={rgb, 255:red, 208; green, 2; blue, 27 }  ,draw opacity=1 ]   (405.23,194.59) .. controls (402.91,136.55) and (317.32,246.38) .. (330.25,139.9) ;
\draw [shift={(330.45,138.28)}, rotate = 97.31] [color={rgb, 255:red, 208; green, 2; blue, 27 }  ,draw opacity=1 ][line width=0.75]    (10.93,-3.29) .. controls (6.95,-1.4) and (3.31,-0.3) .. (0,0) .. controls (3.31,0.3) and (6.95,1.4) .. (10.93,3.29)   ;
\draw [color={rgb, 255:red, 208; green, 2; blue, 27 }  ,draw opacity=1 ] [dash pattern={on 0.84pt off 2.51pt}]  (303.65,132.48) .. controls (310.99,154.15) and (319.79,116.95) .. (330.45,138.28) ;
\draw [color={rgb, 255:red, 208; green, 2; blue, 27 }  ,draw opacity=1 ]   (303.3,128.99) .. controls (301.33,115.53) and (312.12,124.99) .. (306.71,90.11) ;
\draw [shift={(306.45,88.48)}, rotate = 80.74] [color={rgb, 255:red, 208; green, 2; blue, 27 }  ,draw opacity=1 ][line width=0.75]    (10.93,-3.29) .. controls (6.95,-1.4) and (3.31,-0.3) .. (0,0) .. controls (3.31,0.3) and (6.95,1.4) .. (10.93,3.29)   ;
\draw [color={rgb, 255:red, 208; green, 2; blue, 27 }  ,draw opacity=1 ] [dash pattern={on 0.84pt off 2.51pt}]  (306.45,88.48) .. controls (324.25,74.33) and (304.92,104.75) .. (333.25,118.08) ;
\draw [color={rgb, 255:red, 208; green, 2; blue, 27 }  ,draw opacity=1 ]   (333.25,118.08) .. controls (359.12,80.07) and (388.56,163.44) .. (427.86,57.49) ;
\draw [shift={(428.45,55.88)}, rotate = 110.07] [color={rgb, 255:red, 208; green, 2; blue, 27 }  ,draw opacity=1 ][line width=0.75]    (10.93,-3.29) .. controls (6.95,-1.4) and (3.31,-0.3) .. (0,0) .. controls (3.31,0.3) and (6.95,1.4) .. (10.93,3.29)   ;
\draw [color={rgb, 255:red, 208; green, 2; blue, 27 }  ,draw opacity=1 ]   (283.95,213.48) .. controls (333.15,181.08) and (314.05,125.88) .. (311.25,121.08) .. controls (308.45,116.28) and (316.05,108.08) .. (322.45,103.08) .. controls (328.85,98.08) and (339.89,80.84) .. (346.45,84.48) .. controls (353.01,88.13) and (369.61,105.25) .. (374.89,106.55) .. controls (380.18,107.85) and (396.17,99.28) .. (395.15,96.59) .. controls (394.16,94) and (333.64,61.2) .. (318.5,54.75) ;
\draw [shift={(316.85,54.08)}, rotate = 20.42] [color={rgb, 255:red, 208; green, 2; blue, 27 }  ,draw opacity=1 ][line width=0.75]    (10.93,-3.29) .. controls (6.95,-1.4) and (3.31,-0.3) .. (0,0) .. controls (3.31,0.3) and (6.95,1.4) .. (10.93,3.29)   ;

\draw (117.79,192.8) node [anchor=north west][inner sep=0.75pt]  [color={rgb, 255:red, 208; green, 2; blue, 27 }  ,opacity=1 ] [align=left] {$\displaystyle \tilde{D}_{a_{n}}{} \ $};
\draw (51.66,226.4) node [anchor=north west][inner sep=0.75pt]  [color={rgb, 255:red, 208; green, 2; blue, 27 }  ,opacity=1 ] [align=left] {$\displaystyle D_{a_{n}}{} \ $};
\draw (160.08,169.08) node [anchor=north west][inner sep=0.75pt]   [align=left] {$\displaystyle (\hat{\Sigma } ,D)$};
\draw (207.54,68.36) node [anchor=north west][inner sep=0.75pt]   [align=left] {$\displaystyle \hat{\mathcal{U}}_{a_{n}}{} \ $};
\draw (393.26,193.34) node [anchor=north west][inner sep=0.75pt]  [color={rgb, 255:red, 208; green, 2; blue, 27 }  ,opacity=1 ] [align=left] {$\displaystyle V( D)_{a_{n}}{} \ $};
\draw (259.79,208.81) node [anchor=north west][inner sep=0.75pt]  [color={rgb, 255:red, 208; green, 2; blue, 27 }  ,opacity=1 ] [align=left] {$\displaystyle \widetilde{V( D)}_{a_{n}}{} \ $};
\draw (423.42,156.41) node [anchor=north west][inner sep=0.75pt]   [align=left] {$\displaystyle (\hat{\Sigma } ,V( D))$};
\draw (468.46,57.24) node [anchor=north west][inner sep=0.75pt]   [align=left] {$\displaystyle \hat{\mathcal{U}}_{a_{n}}{} \ $};

\end{tikzpicture}

      \caption{  Construction of the primitive curves of a virtualization of a doodle diagram}
        \label{primivirt}
    \end{figure}
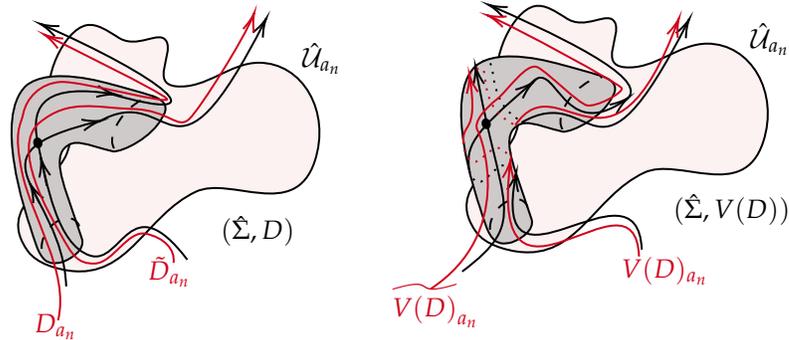
    
\end{proof}
Up to permutation, we may assume, without loss of generality that the virtualization is given in the crossing point $a_n$. In order to make the notation easier, we introduce the following definition. 
\begin{definition}
    Let us consider two transverse, oriented and generically immersed  curves $\gamma_1$ and $\gamma_2$ on a surface $\Sigma$ and let $\mathcal{U}$ be a subset of $\Sigma$. We define the \textbf{relative intersection pairing number}, denoted by $\left\langle \gamma_1,\gamma_2 \right\rangle\mid_{\mathcal{U}}$ as the number
   \begin{equation}
    \left\langle \gamma_1,\gamma_2 \right\rangle \mid_{\mathcal{U}}=\displaystyle \sum\limits_{c_k^{e_k}\in \gamma_1 \nearrow_{\mathcal{U}} \gamma_2 }e_k,
    \label{releq}
\end{equation}
    where $\gamma_1 \nearrow_{\mathcal{U}} \gamma_2=\{c_i^{e_i}\in \gamma_1 \nearrow \gamma_2 \mid c_i \in \gamma_1 \cap \gamma_2 \cap \mathcal{U}\}$.
\end{definition}
The relative intersection pairing number is well defined in $H_1(\mathcal{U},\mathbf{Z})$. We now consider the equivalence classes, denoted by $V(\varphi_D), V(\varphi)_{a_1},\ldots,V(\varphi)_{a_n}$,  in the homology group $H_1(\widehat{\Sigma},\mathbf{Z})$ represented by the curves $V(D),V(D)_{a_1},\ldots,V(D)_{a_n}$, respectively. Thereby, from Proposition \ref{comple} and Equation \ref{releq}, 
\begin{equation}
     \left\langle \varphi_{a_i},\varphi_{D} \right\rangle \mid_{\widehat{\Sigma}\setminus \widehat{\mathcal{U}}_{a_n}} =\left\langle V(\varphi)_{a_i},V(\varphi_{D}) \right\rangle \mid_{\widehat{\Sigma}\setminus \widehat{\mathcal{U}}_{a_n}}
    \text{ and }
    \left\langle \varphi_{a_i},\varphi_{a_j} \right\rangle \mid_{\widehat{\Sigma}\setminus \widehat{\mathcal{U}}_{a_n}} =\left\langle V(\varphi)_{a_i},V(\varphi)_{a_j} \right\rangle \mid_{\widehat{\Sigma}\setminus \widehat{\mathcal{U}}_{a_n}},
    \label{res02}
\end{equation}
for every $i,j\neq n$. We also have,
\begin{equation}
     \left\langle \widetilde{\varphi}_{a_n},\varphi_{D} \right\rangle \mid_{\widehat{\Sigma}\setminus \widehat{\mathcal{U}}_{a_n}} =\left\langle V(\varphi)_{a_n},V(\varphi_{D}) \right\rangle \mid_{\widehat{\Sigma}\setminus \widehat{\mathcal{U}}_{a_n}}.
    \label{res03}
\end{equation}
\begin{theorem}\label{theo01}
    With the above notation. Let $a_i \neq a_n$.
    \begin{enumerate}
        \item  If $\{a_n,a_{n}^{-1}\}\cap D_{a_i}\nearrow D \in \{\emptyset,\{a_n,a_n^{-1}\}\}$, then $\left\langle V( \varphi)_{a_i},V(\varphi_{D}) \right\rangle=\left\langle \varphi_{a_i},\varphi_{D} \right\rangle$.
        \item If $\{a_n,a_{n}^{-1}\}\cap D_{a_i}\nearrow D =\{a_{n}^{\epsilon}\}$, then $\left\langle V( \varphi)_{a_i},V(\varphi_{D}) \right\rangle=\left\langle \varphi_{a_i},\varphi_{D} \right\rangle-2\epsilon$.
    \end{enumerate}
    On the other hand, $\left\langle V( \varphi)_{a_n},V(\varphi_{D}) \right\rangle=-\left\langle \varphi_{a_n},\varphi_{D} \right\rangle$.
\end{theorem}
\begin{proof}
 Let $a_i\neq a_n$.
 \begin{enumerate}
     \item We suppose that $\{a_n,a_{n}^{-1}\}\cap D_{a_i}\nearrow D \in \{\emptyset,\{a_n,a_n^{-1}\}\}$. Then, from the definition of virtualization, we have $\{a_n,a_{n}^{-1}\}\cap D_{a_i}\nearrow D =\{a_n,a_{n}^{-1}\}\cap V(D)_{a_i}\nearrow V(D)$. Therefore, relative to the neighborhood $\mathcal{U}_{a_n}$, the equality $\left\langle V( \varphi)_{a_i},V(\varphi_{D}) \right\rangle_{\widetilde{\mathcal{U}}_{a_n}}=\left\langle \varphi_{a_i},\varphi_{D} \right\rangle_{\widetilde{\mathcal{U}}_{a_n}}$ is satisfied. Thus, from the first equation in  (\ref{res02}), we have that  $\left\langle V( \varphi)_{a_i},V(\varphi_{D}) \right\rangle=\left\langle \varphi_{a_i},\varphi_{D} \right\rangle$.
     \item Suppose that $\{a_n,a_{n}^{-1}\}\cap D_{a_i}\nearrow D =\{a_{n}^{\epsilon}\}$. Then, from the definition of virtualization, we have $\{a_n,a_{n}^{-1}\}\cap V(D)_{a_i}\nearrow V(D) =\{a_{n}^{-\epsilon}\}$. Thus, $\left\langle V( \varphi)_{a_i},V(\varphi_{D}) \right\rangle_{\widetilde{\mathcal{U}}_{a_n}}=-\epsilon$ and $\left\langle \varphi_{a_i},\varphi_{D} \right\rangle_{\widetilde{\mathcal{U}}_{a_n}}=\epsilon$. Therefore, from the first equation in  (\ref{res02}), we have that  $\left\langle V( \varphi)_{a_i},V(\varphi_{D}) \right\rangle=\left\langle \varphi_{a_i},\varphi_{D} \right\rangle-2\epsilon$.
 \end{enumerate}
Besides, from Figure \ref{primivirt}, we have that $\left\langle V( \varphi)_{a_n},V(\varphi_{D}) \right\rangle_{\widetilde{\mathcal{U}}_{a_n}}=0$ and $\left\langle \varphi_{a_n},\varphi_{D} \right\rangle_{\widetilde{\mathcal{U}}_{a_n}}=0$. Thus, from Equation \ref{res03}, we have 
\[\begin{array}{ccl}
  \left\langle V(\varphi)_{a_n},V(\varphi_{D}) \right\rangle    & =&\left\langle V(\varphi)_{a_n},V(\varphi_{D}) \right\rangle \mid_{\widehat{\Sigma}\setminus \widehat{\mathcal{U}}_{a_n}} + \left\langle V(\varphi)_{a_n},V(\varphi_{D}) \right\rangle \mid_{ \widehat{\mathcal{U}}_{a_n}}\\
     &&\\
     &=&\left\langle \widetilde{\varphi}_{a_n},\varphi_{D} \right\rangle \mid_{\widehat{\Sigma}\setminus \widehat{\mathcal{U}}_{a_n}} + \left\langle \widetilde{\varphi}_{a_n},\varphi_{D} \right\rangle \mid_{ \widehat{\mathcal{U}}_{a_n}}\\
     &&\\
     &=&\left\langle \widetilde{\varphi}_{a_n},\varphi_{D} \right\rangle. 
\end{array}
\]
Thus, from Equation (\ref{eq:rela}), $\left\langle V(\varphi)_{a_n},V(\varphi_{D}) \right\rangle=-\left\langle \varphi_{a_n},\varphi_{D} \right\rangle.$
\end{proof}

We consider the following theorem that involves the behaviour of the entries of the matrix $\beta(D)$ of a one-component doodle diagram $(\Sigma,D)$.
\begin{theorem}\label{theo02}
   Let $(\Sigma,D)$ be a one-component doodle diagram and let $a_i\neq a_n$. 
   \begin{enumerate}
        \item  If $\{a_n,a_{n}^{-1}\}\cap D_{a_i}\nearrow D =\emptyset$, then $\left\langle V( \varphi)_{a_i},V(\varphi)_{a_j} \right\rangle=\left\langle \varphi_{a_i},\varphi_{a_j} \right\rangle$, for every $j=1,\ldots,n$.
        \item  Suppose that $\{a_n,a_{n}^{-1}\}\cap D_{a_i}\nearrow D =\{a_n,a_{n}^{-1}\}$, we have the following cases
        \begin{itemize}
            \item  $\{a_n,a_{n}^{-1}\}\cap D_{a_j}\nearrow D \in \{\emptyset,\{a_n,a_{n}^{-1}\}\}$, then $\left\langle V( \varphi)_{a_i},V(\varphi)_{a_j} \right\rangle=\left\langle \varphi_{a_i},\varphi_{a_j} \right\rangle$.
            \item $\{a_n,a_{n}^{-1}\}\cap D_{a_j}\nearrow D =\{a_{n}^{\epsilon}\}$, then $\left\langle V( \varphi)_{a_i},V(\varphi)_{a_j} \right\rangle=\left\langle \varphi_{a_i},\varphi_{a_j} \right\rangle+2\epsilon$.
            \item $\left\langle V( \varphi)_{a_i},V(\varphi)_{a_n} \right\rangle=\left\langle \varphi_{a_i},\varphi_{D} \right\rangle+\left\langle \varphi_{a_i},\varphi_{a_n} \right\rangle$.
        \end{itemize}
      \item    Suppose that $\{a_n,a_{n}^{-1}\}\cap D_{a_i}\nearrow D =\{a_n^{\epsilon}\}$. 
      \begin{itemize}
            \item $\{a_n,a_{n}^{-1}\}\cap D_{a_j}\nearrow D =\{a_{n},a_n^{-1}\}$, then $\left\langle V( \varphi)_{a_i},V(\varphi)_{a_j} \right\rangle=\left\langle \varphi_{a_i},\varphi_{a_j} \right\rangle-2\epsilon$.
            \item $\{a_n,a_{n}^{-1}\}\cap D_{a_i}\nearrow D =\{a_n^{\epsilon}\}$, then $\left\langle V( \varphi)_{a_i},V(\varphi)_{a_j} \right\rangle=\left\langle \varphi_{a_i},\varphi_{a_j} \right\rangle$.
            \item $\{a_n,a_{n}^{-1}\}\cap D_{a_i}\nearrow D =\{a_n^{-\epsilon}\}$, then $\left\langle V( \varphi)_{a_i},V(\varphi)_{a_j} \right\rangle=\left\langle \varphi_{a_i},\varphi_{a_j} \right\rangle-2\epsilon$.
            \item $\left\langle V( \varphi)_{a_i},V(\varphi)_{a_n} \right\rangle=\left\langle \varphi_{a_i},\varphi_{D} \right\rangle+\left\langle \varphi_{a_i},\varphi_{a_n} \right\rangle-\epsilon$.
        \end{itemize}
   \end{enumerate}
\end{theorem}
\begin{proof}
 Let $a_i\neq a_n$.
 \begin{enumerate}
     \item Suppose that $\{a_n,a_{n}^{-1}\}\cap D_{a_i}\nearrow D =\emptyset$. Then, $D_{a_i}\cap \widehat{\mathcal{U}}_{a_n}=\emptyset$. Therefore, $V(D)_{a_i}=D_{a_i}$ and $\left\langle \varphi_{a_i},\varphi_{a_j}\right\rangle\mid_{\widehat{\mathcal{U}}_{a_n}}=0$. We also have $\left\langle V(\varphi)_{a_i},V(\varphi)_{a_j}\right\rangle\mid_{\widehat{\mathcal{U}}_{a_n}}=0$. Thus,  $\left\langle \varphi_{a_i},\varphi_{a_j}\right\rangle\mid=\left\langle V(\varphi)_{a_i},V(\varphi)_{a_j}\right\rangle$. 
     \item Suppose that $\{a_n,a_{n}^{-1}\}\cap D_{a_i}\nearrow D =\{a_n,a_{n}^{-1}\}$. Then, $\{a_n,a_{n}^{-1}\}\cap V(D)_{a_i}\nearrow V(D) =\{a_n,a_{n}^{-1}\}$. Moreover,
      \begin{itemize}
            \item If $\{a_n,a_{n}^{-1}\}\cap D_{a_j}\nearrow D=\emptyset$, we use literal $(a)$. But, if $\{a_n,a_{n}^{-1}\}\cap D_{a_j}\nearrow D=\{a_n,a_{n}^{-1}\}$, then $\left\langle V( \varphi)_{a_i},V(\varphi)_{a_j} \right\rangle\mid_{\widetilde{\mathcal{U}}_{a_n}}=\left\langle \varphi_{a_i},\varphi_{a_j} \right\rangle\mid_{\widetilde{\mathcal{U}}_{a_n}}=0$. As a consequence, 
            \[
            \begin{array}{ccl}
                \left\langle V( \varphi)_{a_i},V(\varphi)_{a_j} \right\rangle & =&\left\langle V( \varphi)_{a_i},V(\varphi)_{a_j} \right\rangle\mid_{\widehat{\Sigma}\setminus \widetilde{\mathcal{U}}_{a_n}}+\left\langle V( \varphi)_{a_i},V(\varphi)_{a_j} \right\rangle\mid_{\widetilde{\mathcal{U}}_{a_n}} \\
                 & &\\
                 &=&\left\langle V( \varphi)_{a_i},V(\varphi)_{a_j} \right\rangle\mid_{\widehat{\Sigma}\setminus \widetilde{\mathcal{U}}_{a_n}}\\
                 &&\\
                 &=&\left\langle \varphi_{a_i},\varphi_{a_j} \right\rangle\mid_{\widehat{\Sigma}\setminus \widetilde{\mathcal{U}}_{a_n}}=\left\langle \varphi_{a_i},\varphi_{a_j} \right\rangle.
            \end{array}
            \]
            \item If $\{a_n,a_{n}^{-1}\}\cap D_{a_j}\nearrow D =\{a_{n}^{\epsilon}\}$, then $\left\langle \varphi_{a_i},\varphi_{a_j} \right\rangle\mid_{\widetilde{\mathcal{U}}_{a_n}}=-\epsilon$. Thereby, $\left\langle V(\varphi)_{a_i},V(\varphi)_{a_j} \right\rangle\mid_{\widetilde{\mathcal{U}}_{a_n}}=\epsilon$. Thus,
            \[
            \begin{array}{ccl}
                \left\langle V( \varphi)_{a_i},V(\varphi)_{a_j} \right\rangle & =&\left\langle V( \varphi)_{a_i},V(\varphi)_{a_j} \right\rangle\mid_{\widehat{\Sigma}\setminus \widetilde{\mathcal{U}}_{a_n}}+\left\langle V( \varphi)_{a_i},V(\varphi)_{a_j} \right\rangle\mid_{\widetilde{\mathcal{U}}_{a_n}} \\
                 & &\\
                 &=&\left\langle V( \varphi)_{a_i},V(\varphi)_{a_j} \right\rangle\mid_{\widehat{\Sigma}\setminus \widetilde{\mathcal{U}}_{a_n}}+\epsilon\\
                 &&\\
                 &=&\left\langle \varphi_{a_i},\varphi_{a_j} \right\rangle\mid_{\widehat{\Sigma}\setminus \widetilde{\mathcal{U}}_{a_n}}+\epsilon\\
                 &&\\
                 &=&\left\langle \varphi_{a_i},\varphi_{a_j} \right\rangle\mid_{\widehat{\Sigma}\setminus \widetilde{\mathcal{U}}_{a_n}}-\epsilon +2\epsilon\\
                 &&\\
                 &=&\left\langle \varphi_{a_i},\varphi_{a_j} \right\rangle+2\epsilon.
            \end{array}
            \]
            \item  From Figure \ref{primivirt}, $\left\langle \varphi_{a_i},\widetilde{\varphi}_{a_n} \right\rangle\mid_{\widetilde{\mathcal{U}}_{a_n}}=0$ and $\left\langle V(\varphi)_{a_i},V(\varphi)_{a_n} \right\rangle\mid_{\widetilde{\mathcal{U}}_{a_n}}=0$. Then, 
            \[
            \begin{array}{ccl}
                \left\langle V( \varphi)_{a_i},V(\varphi)_{a_n} \right\rangle & =&\left\langle V( \varphi)_{a_i},V(\varphi)_{a_n} \right\rangle\mid_{\widehat{\Sigma}\setminus \widetilde{\mathcal{U}}_{a_n}}+\left\langle V( \varphi)_{a_i},V(\varphi)_{a_n} \right\rangle\mid_{\widetilde{\mathcal{U}}_{a_n}} \\
                 & &\\
                 &=&\left\langle V( \varphi)_{a_i},V(\varphi)_{a_n} \right\rangle\mid_{\widehat{\Sigma}\setminus \widetilde{\mathcal{U}}_{a_n}}=\left\langle \varphi_{a_i},\widetilde{\varphi}_{a_n} \right\rangle\mid_{\widehat{\Sigma}\setminus \widetilde{\mathcal{U}}_{a_n}}\\
                 &&\\
                 &=&\left\langle \varphi_{a_i},\widetilde{\varphi}_{a_n} \right\rangle=\left\langle \varphi_{a_i},\varphi_{D} \right\rangle+\left\langle \varphi_{a_i},\varphi_{a_n} \right\rangle.
            \end{array}
            \]
        \end{itemize}
        \item Suppose that $\{a_n,a_{n}^{-1}\}\cap D_{a_i}\nearrow D =\{a_n^{\epsilon}\}$. Then, $\{a_n,a_{n}^{-1}\}\cap V(D)_{a_i}\nearrow V(D) =\{a_n^{-\epsilon}\}$. Moreover, 
      \begin{itemize}
            \item If $\{a_n,a_{n}^{-1}\}\cap D_{a_j}\nearrow D =\{a_{n},a_n^{-1}\}$. From literal $(a)$,  $\left\langle V( \varphi)_{a_j},V(\varphi)_{a_i} \right\rangle=\left\langle \varphi_{a_j},\varphi_{a_i} \right\rangle+2\epsilon$. Therefore, $\left\langle V( \varphi)_{a_i},V(\varphi)_{a_j} \right\rangle=\left\langle \varphi_{a_i},\varphi_{a_j} \right\rangle-2\epsilon$.
            \item If $\{a_n,a_{n}^{-1}\}\cap D_{a_i}\nearrow D =\{a_n^{\epsilon}\}$, then $\left\langle \varphi_{a_i},\varphi_{a_j}\right\rangle\mid_{\widehat{\mathcal{U}}_{a_n}}=0$. We also have the equality $\{a_n,a_{n}^{-1}\}\cap V(D)_{a_i}\nearrow V(D) =\{a_n^{-\epsilon}\}$. Thus, $\left\langle V(\varphi)_{a_i},V(\varphi)_{a_j}\right\rangle\mid_{\widehat{\mathcal{U}}_{a_n}}=0$. Therefore, $\left\langle V( \varphi)_{a_i},V(\varphi)_{a_j} \right\rangle=\left\langle \varphi_{a_i},\varphi_{a_j} \right\rangle$.
            \item If $\{a_n,a_{n}^{-1}\}\cap D_{a_i}\nearrow D =\{a_n^{-\epsilon}\}$, then $\left\langle  \varphi_{a_i},\varphi_{a_j} \right\rangle_{\widehat{\mathcal{U}}_{a_n}}=\epsilon$. Besides, we also have the equality $\{a_n,a_{n}^{-1}\}\cap V(D)_{a_i}\nearrow V(D) =\{a_n^{\epsilon}\}$. Thus, $\left\langle  V(\varphi)_{a_i},V(\varphi)_{a_j} \right\rangle_{\widehat{\mathcal{U}}_{a_n}}=-\epsilon$. Therefore, 
             \[
            \begin{array}{ccl}
                \left\langle V( \varphi)_{a_i},V(\varphi)_{a_j} \right\rangle & =&\left\langle V( \varphi)_{a_i},V(\varphi)_{a_j} \right\rangle\mid_{\widehat{\Sigma}\setminus \widetilde{\mathcal{U}}_{a_n}}+\left\langle V( \varphi)_{a_i},V(\varphi)_{a_j} \right\rangle\mid_{\widetilde{\mathcal{U}}_{a_n}} \\
                 & &\\
                 &=&\left\langle V( \varphi)_{a_i},V(\varphi)_{a_j} \right\rangle\mid_{\widehat{\Sigma}\setminus \widetilde{\mathcal{U}}_{a_n}}-\epsilon\\
                 &&\\
                 &=&\left\langle \varphi_{a_i},\varphi_{a_j} \right\rangle\mid_{\widehat{\Sigma}\setminus \widetilde{\mathcal{U}}_{a_n}}+\epsilon-2\epsilon\\
                 &&\\
                 &=&\left\langle \varphi_{a_i},\varphi_{a_j} \right\rangle\mid_{\widehat{\Sigma}\setminus \widetilde{\mathcal{U}}_{a_n}}+\left\langle  \varphi_{a_i},\varphi_{a_j} \right\rangle_{\widehat{\mathcal{U}}_{a_n}} -2\epsilon\\
                 &&\\
                 &=&\left\langle \varphi_{a_i},\varphi_{a_j} \right\rangle-2\epsilon.
            \end{array}
            \]
            \item Since $\{a_n,a_{n}^{-1}\}\cap D_{a_i}\nearrow D =\{a_n^{\epsilon}\}$, then, $\left\langle  \varphi_{a_i},\widetilde{\varphi}_{a_n} \right\rangle_{\widehat{\mathcal{U}}_{a_n}}=\epsilon$. Besides, we know that $\left\langle  V(\varphi)_{a_i},V(\varphi)_{a_n} \right\rangle_{\widehat{\mathcal{U}}_{a_n}}=0$. Thus, 

            \[
            \begin{array}{ccl}
                \left\langle V( \varphi)_{a_i},V(\varphi)_{a_n} \right\rangle & =&\left\langle V( \varphi)_{a_i},V(\varphi)_{a_n} \right\rangle\mid_{\widehat{\Sigma}\setminus \widetilde{\mathcal{U}}_{a_n}}+\left\langle V( \varphi)_{a_i},V(\varphi)_{a_n} \right\rangle\mid_{\widetilde{\mathcal{U}}_{a_n}} \\
                 & &\\
                 &=&\left\langle V( \varphi)_{a_i},V(\varphi)_{a_n} \right\rangle\mid_{\widehat{\Sigma}\setminus \widetilde{\mathcal{U}}_{a_n}}=\left\langle \varphi_{a_i},\widetilde{\varphi}_{a_n} \right\rangle\mid_{\widehat{\Sigma}\setminus \widetilde{\mathcal{U}}_{a_n}}\\
                 &&\\
                 &=&\left\langle \varphi_{a_i},\widetilde{\varphi}_{a_n} \right\rangle\mid_{\widehat{\Sigma}\setminus \widetilde{\mathcal{U}}_{a_n}}+\left\langle \varphi_{a_i},\widetilde{\varphi}_{a_n} \right\rangle\mid_{ \widetilde{\mathcal{U}}_{a_n}} -\epsilon\\
                 &&\\
                 &=&\left\langle \varphi_{a_i},\widetilde{\varphi}_{a_n} \right\rangle-\epsilon=\left\langle \varphi_{a_i},\varphi_{D} \right\rangle+\left\langle \varphi_{a_i},\varphi_{a_n} \right\rangle-\epsilon.
            \end{array}
            \]
        \end{itemize}
 \end{enumerate}
\end{proof}

A direct consequence of  Theorems \ref{theo01} and \ref{theo02} is the following corollary.
\begin{corollary}
      Let $(\Sigma,D)$ be an almost-classical doodle diagram. Then $\lambda((\Sigma,V(D)))$ is equivalent to

\begin{equation*}
 \left( \begin{array}{rrrrrrr|r}

0&\dots& 0&-2&\dots& -2&-1&-2\\
\vdots&&\vdots&\vdots&&\vdots&\vdots&\vdots\\
0&\dots& 0&-2&\dots& -2&-1&-2\\

2&\dots& 2&0&\dots& 0&1&2\\
\vdots&&\vdots&\vdots&&\vdots&\vdots&\vdots\\
2&\dots& 2&0&\dots& 0&1&2\\
1&\dots&1&-1&\dots&-1&0&0

\end{array}\right).
\end{equation*}
Where the block of $2$'s is of size $k \times k$ and the block of $-2$'s is of size $m \times m$, with $k$ (and $m$) representing the cardinality of the set of all $a_j$ for which $D_{a_j}$ crosses the point $a_n$ once in a positive (negative) way. Thus, if $k=m$, then $\lambda((\Sigma,V(D))))$ is equivalent to $\left(\ \ \mid \ \ \right)$.
    
\end{corollary}

\end{document}